\documentclass[11pt]{article}

\usepackage{amsfonts,amsthm,amsmath,amssymb}
\usepackage{graphicx}
\usepackage{mathtools,booktabs,bm,bbm,stmaryrd}
\usepackage{color}

\newcommand{\enorm}   [1] {\interleave #1 \interleave}
\newcommand{\norm}   [1] {\Vert#1\Vert}
\newcommand{\jump}   [1] {[\![#1]\!]}
\newcommand{\wph}{\bm{W}}
\newcommand{\uph}{U}
\newcommand{\pph}{P}
\newcommand{\qph}{Q}
\newcommand{\vph}{\bm \chi}
\newcommand{\dual}[2]{\langle#1,#2\rangle}
\newcommand{\yao}[1]{ {\color{red}{#1}} }

\theoremstyle{dgthm}
\newtheorem{theorem}{Theorem}

\newtheorem{proposition}{Proposition}
\newtheorem{lemma}{Lemma}

\theoremstyle{dgdef}

\newtheorem{remark}{Remark}

\title{Local discontinuous Galerkin method for a third order singularly perturbed problem of convection-diffusion type}

\author{
	Li Yan\footnotemark[1]
	\quad
	Zhoufeng Wang\footnotemark[2]
	\quad
	Yao Cheng\footnotemark[1]
}

\begin{document}
	
\maketitle

\footnotetext[1]
{School of Mathematical Sciences,
	Suzhou University of Science and Technology,
	Suzhou 215009, Jiangsu Province, China;
	Research supported by NSFC grant 11801396, Natural Science Foundation of Jiangsu Province grant BK20170374, Natural Science Foundation of the Jiangsu Higher Education Institutions of China grant 17KJB110016;
	\textbf{yanl@post.usts.edu.cn;ycheng@usts.edu.cn} 
}
\footnotetext[2]
{Henan University of Science and Technology, 
	School of Mathematics and Statistics,
	Henan, China; \textbf{ zfwang801003@126.com}
}

\begin{abstract}
	The local discontinuous Galerkin (LDG) method is
	studied for a third-order singularly perturbed problem
	of the convection-diffusion type.
	Based on a regularity assumption for the exact solution,
	we prove almost $O(N^{-(k+1/2)})$ (up to a logarithmic factor)
		energy-norm convergence uniformly in the perturbation parameter.
		Here, $k\geq 0$ is the maximum degree of piecewise polynomials
		used in discrete space, and $N$ is the number of mesh elements.
	The results are valid for the three types of layer-adapted  meshes:
	Shishkin-type, Bakhvalov-Shishkin type, and Bakhvalov-type.
	Numerical experiments are conducted to test the theoretical results.
\end{abstract}


\section{Introduction}
Singularly perturbed problems have arisen 
frequently in fluid mechanics, elasticity, chemical reactor theory,
	and many other related areas \cite{Roos2008}.
Second- and fourth-order singularly perturbed problems
have been widely studied.
Only a few results have been reported
for third-order singularly perturbed problems,
which might come from
the theory of dispersive systems and thin-film flows;
see the applications described in \cite{Howes1983,Howes1983:Stud}.
Consider the third-order singularly perturbed problem, 
\begin{equation}\label{SPP:3RD}
\begin{split}
\varepsilon u^{\prime \prime \prime}(x)
- (a(x)u^{\prime}(x))^{\prime}+ b(x)u^{\prime}(x) + c(x)u(x) =& f(x)
\hspace{1em}{\rm in}\hspace{0.5em}\Omega  = (0,1),\\
u(0) = u(1) = u^{\prime}(1) =& 0,
\end{split}
\end{equation}
where $0 < \varepsilon  \ll 1$ is the perturbation parameter,
and $a,b,c,f$ are smooth functions 
on $\overline{\Omega}$ that satisfy
\begin{equation}\label{assumption:coefficient}
a(x) \ge \alpha  > 0,\quad 
c(x) -\frac12b^{\prime}(x) \ge \gamma  > 0, \quad x\in \Omega
\end{equation}
for the positive constants $\alpha$ and $\gamma$.
The solution to problem \eqref{SPP:3RD} typically
has a weak boundary layer at $x=1$ (see \eqref{regularity}).


In \cite{Roos2015},
Roos et al. employed an upwind difference scheme
on a Shishkin mesh to solve an analogous third-order singularly perturbed problem.
They obtained an almost first-order uniform convergence.
In \cite{Valarmathi2002},
Valarmathi and Ramanujam
transformed the third-order equation into a weakly
coupled system of first- and second-order equations.
An exponentially fitted difference scheme
and a classical difference scheme were combined
to solve the following two equations
in fine and rough domains, respectively.
Similarly, in \cite{Temsah2008},
Temsah obtained some numerical results
to test the spectral collocation method
for the third-order singularly perturbed problem
of convection-diffusion type and reaction-diffusion type.
As for finite element discretization,
Zarin et al. studied
a $C^0$-continuous interior penalty method
and obtained uniform convergence of order $k-1$
in energy-norm
for three Shishkin-type meshes \cite{Zarin2016},
where $k>1$ is the highest degree of piecewise polynomials.

This study aims to develop
a well-known version of the discontinuous Galerkin (DG) method,
i.e. the local discontinuous Galerkin (LDG) method
for the problem \eqref{SPP:3RD}.
This method is regarded as a successful application
of the DG method to a convection–diffusion system \cite{Cockburn:Shu:LDG}.
The main idea is to rewrite the second-order equation
into an equivalent first-order system
and then apply the DG method to solve each differential equation.
The LDG method inherits the advantages of the DG method,
such as local conservativity,
flexibility with the mesh-design,
and the possibility for adaptive $hp$-strategy.
Thus, it is suited to problems where solutions have steep gradients
	or boundary layers. In the past 20 years, there have been extensive studies
	on LDG methods for solving various equations
	with higher order derivatives.


For the second order singularly perturbed problem,
	the LDG method produces good results, even on a uniform mesh;
	see \cite{CYWL22,Xie2009JCM}.
In \cite{CYWL22,Cheng:Zhang2017},
Cheng et al. demonstrated the double-optimal local error estimate of the LDG method
on quasi-uniform meshes.
In \cite{Zhu:NMPDE:2013,Zhu:2dMC},
Zhu and Zhang studied the uniform convergence of the LDG method
on the standard Shishkin mesh.
Recently, the LDG method was applied to solve a two-parameter singularly perturbed problem, and several error estimates were obtained
on six types of layer-adapted meshes in a uniform framework \cite{Cheng2021}.
Uniform convergence of the LDG method
with generalised alternating numerical fluxes
was also obtained
for a two-dimensional singularly perturbed convection-diffusion problem
\cite{Cheng2020}.

However, we are not aware of any results of the LDG method
for third-order singularly perturbed problems.
In this paper,
we study the LDG method
for a third order singularly perturbed problem \eqref{SPP:3RD}
of convection-diffusion type.
In this study, several aspects were addressed.
\begin{itemize}
	\item 
Although the LDG method has been studied for many high-order differential equations, no LDG scheme is available in the literature
for third-order singularly perturbed problem \eqref{SPP:3RD}.
To construct the LDG scheme,
one has to design the numerical fluxes carefully
which obey some known principles, such as upwinding
for the convective flux, alternating for the diffusive flux,
and alternating for the dispersive flux.
However, the way to coordinate them suitably,
including a suitable setting of these fluxes on the domain boundaries,
is unclear.
For the first time, in this study, we propose a stable and high-accuracy LDG scheme
for problem \eqref{SPP:3RD}.

\item
Uniform convergence is often performed 
for the numerical method on the piecewise
uniform Shishkin mesh (S-mesh).
However, some generalizations of layer-adapted meshes \cite{Linss10},
such as Bakhvalov-Shishkin mesh (BS-mesh) and Bakhvalov-type mesh
(B-mesh), are alternatives to eliminate the influence of the
logarithmic factor on the convergence rate.
In this study, we carry out error analysis on three typical
layer-adapted meshes, namely, an S-mesh,
a BS-mesh, and a B-mesh.
Based on the regularity assumption of the exact solution of \eqref{SPP:3RD}
and the approximation errors of the local Gauss-Radau projection
on the aforementioned layer-adapted meshes \cite{Cheng2021},
we prove an almost optimal energy-norm error estimate of the LDG method
on these meshes in a unified framework.
\item
In the theoretical analysis, 
we use the local Gauss-Radau projections to address the projection error. However, a difficulty arises from the estimate of the projection error about the second-order derivative.
Because we have no stability about this term in the energy-norm of the projection error,
we have to seek its control by the first-order derivative 
as well as the element interface jump
of projection error about the primal variable in the energy norm.
This relationship can be established from
the inherent structure of the LDG scheme.
Additionally, to obtain the uniform convergence of 
the LDG method on the B-mesh,
the special structure of this mesh must be studied.	
\end{itemize}
To the best of our knowledge, this is the first uniform-convergence analysis
of the LDG method for third-order singularly perturbed problems
of convection-diffusion type.
These results can also be extended
to the third-order problem of reaction-diffusion type,
even-order singularly perturbed problems,
and higher dimensional problems.
This will appear in future work.

The remainder of this paper is organised as follows.
In Section \ref{Sec:scheme},
we present layer-adapted meshes and describe the LDG method.
In Section \ref{Sec:error},
we carry out the error analysis.
	Approximation properties of the local Gauss-Radau projections
	on layer-adapted meshes are presented in Section \ref{Sec:projection}
	and the main result of the energy-norm error estimate
	is stated in Section \ref{Sec:main:result}.
In Section \ref{Sec:experiments}, 
we present some numerical experiments
to test our theoretical results.
Finally, in Section \ref{Sec:concluding}
we give our concluding remarks, and in the Appendix we prove a technical lemma.


\section{Layer-adapted meshes and the LDG method}
\label{Sec:scheme}

	Assume that the reduced problem of \eqref{SPP:3RD} for 
	$\varepsilon \rightarrow 0$, which is defined as
	\begin{equation}\label{reduced:problem}
	\begin{split}
	- (a(x)z^{\prime}(x))^{\prime}+ b(x)z^{\prime}(x) + c(x)z(x) =& f(x)
	\hspace{1em} x\in \Omega,\\
	z(0) = z(1) =& 0
	\end{split}
	\end{equation}
	is well-defined. Then, for a small $\varepsilon$,
problem \eqref{SPP:3RD} has a unique solution of the form \cite[Theorem 2, Chapter 3]{Malley1974}:
\begin{equation}\label{solution}
u(x)=G(x,\varepsilon)+\varepsilon{\tilde G}(x,\varepsilon)
e^{{\int_1^x{a(t)dt}}/\varepsilon},
\end{equation}
where $G$ and $\tilde G$ have asymptotic power-series expansions in
$\varepsilon $, $G(x,0)=z(x),x\in[0,1]$.
It follows from \eqref{solution} that
\begin{equation}\label{regularity}
|u^{(j)}(x)|\le C(1+\varepsilon^{-j+1}e^{-\alpha(1 - x)/\varepsilon })
\end{equation}
for all $x\in\overline \Omega$ and $j = 0,1,\ldots,\ell$.
Here, $\ell$ is some integer that depends on the regularity of the data.

	\begin{proposition}\cite{Linss2001}\label{Proposition:regularity}
		Let $\ell$ be a non-negative integer. Assume that the problem \eqref{SPP:3RD}
		has a solution which can be decomposed as
		$u=\bar u+u_\varepsilon $,
		where the regular part $\bar u$
		and the layer part $u_\varepsilon$ satisfy
		\begin{equation}\label{u decomposition}
		|{\bar u}^{(j)}|\le C,\hspace{7em}|u_\varepsilon^{(j)}|
		\le C\varepsilon^{1-j}e^{-{\alpha (1-x)}/\varepsilon}
		\end{equation}
		for all $x\in \overline{\Omega}$
		and all nonnegative integers $j=0,1,\ldots,\ell$.	
	\end{proposition}
	
	In the following analysis, we require $\ell=k+3$
	in Proposition \ref{Proposition:regularity},
	where $k$ is the degree of piecewise polynomials in our finite element space.

\subsection{Layer-adapted meshes}
We shall use the layer-adapted meshes as follows.
Let $\varphi:[0,1/2]\to [0,\infty) $ be a mesh-generating function, with $\varphi(0)=0$,
which is continuous, piecewise continuously differentiable,
and monotonically increasing. Let
\begin{align}\label{tau}
\tau  = \min \Big\{ \frac{1}{2},
\frac{\sigma\varepsilon}{\alpha}\varphi\Big(\frac12\Big)\Big\},
\end{align}
where $\sigma > 0$ is a constant to be determined.
Assume that $\tau  =(\sigma\varepsilon/\alpha)\varphi(1/2)$
	as is typically the case for \eqref{SPP:3RD}.

Let $N \ge 2$ be an even integer and $1-\tau $ be a transition point.
Partition $\Omega={\Omega^c}\cup {\Omega^f}$,
where $\Omega^c=[0,1-\tau]$ 
is the coarse domain with $N/2$ equidistant elements
and $\Omega^f=[1-\tau,1]$
is the refined domain with $N/2$ non-uniform elements.
The mesh points $\{x_i\}_{i=0,1,\dots,N}$ are given by
\begin{align}\label{mesh points}
x_i=
\begin{cases}
\frac{2i}{N}(1-\tau)
&{\rm for}\hspace{0.5em}i=0,1,\ldots,N/2-1,\\
1-\frac{\sigma\varepsilon}{\alpha}\varphi(1-\frac{i}{N})
&{\rm for}\hspace{0.5em}
i = N/2,N/2 + 1,\ldots ,N.
\end{cases}
\end{align}

For varying $\varphi$, we obtain different layer-adapted meshes\yao{;}
see the Shishkin mesh (S-mesh), Bakhvalov-Shishkin mesh (BS mesh),
and the Bakhvalov-type mesh (B-mesh) in Table \ref{Table:mesh:functions}.
Here, $\psi =e^{-\varphi}$ is the mesh characterising function,
which plays an important role in our convergence analysis.

\begin{table}[h]
	\centering
	\caption{Three layer-adapted meshes.}
	\label{Table:mesh:functions}
	\begin{tabular}{cccccc}
		\hline
		&$\varphi(t)$&$\min\varphi^{\prime}(t)$&$\max\varphi^{\prime}(t)$&$\psi(t)$&$\max|\psi^{\prime}(t)|$
		\\ 
		\hline
		S-mesh&$2t\ln N$ &$2\ln N$  &$2\ln N$ &$N^{-2t}$&$2\ln N$\\
		BS-mesh&$-\ln[1-2(1-N^{-1})t]$& 2 & $2N$ &$1-2(1-N^{-1})t$&$2$\\
		B-mesh&$-\ln[1-2(1-\varepsilon)t]$ &2 & 2$\varepsilon^{-1}$
		&$1-2(1-\varepsilon)t$&$2$
		\\ 
		\hline
	\end{tabular}
\end{table}

Suppose $\varepsilon \le N^{-1}$, 
meaning we are in the convection-dominated case.
	For each mesh type in Table \ref{Table:mesh:functions},
	we have $\psi(1/2)\leq N^{-1}$.

\begin{lemma}\cite[Lemma 3.1]{Cheng2020}
	For any $j = N/2+1,\ldots,N$, we define:
	\[
	{\mathcal G}_j=\min \Big\{\frac{h_j}{\varepsilon},1\Big\}
	e^{{-\alpha(1-x_j)}/{\sigma\varepsilon}}.
	\]
	Then, there exists a constant $C>0$ independent of $\varepsilon$ and $N$ such that
	\begin{subequations}
		\begin{align}
		\label{inequality G}
		\max_{N/2+1\leq j\leq N}
		\mathcal{G}_j&\le CN^{-1}\max|\psi^{\prime}|,
		\\
		\label{sum inequality G}
		\sum_{j={N/2}+1}^N{\mathcal G}_j&\le C.
		\end{align}
	\end{subequations}
\end{lemma}

\begin{lemma}\label{Lemma:mesh:size}
	For the three meshes in Table \ref{Table:mesh:functions}, we have
	$C \varepsilon N^{-1} \le h_j\le C N^{-1}$. Moreover, for the B-mesh,
	\begin{subequations}
		\begin{align}
		\label{mesh:property:1}
    h_{N/2+j} &\geq \frac{\sigma \varepsilon}{\alpha (j+1)},
\quad j=1,2,\dots,N/2, 
		\\
		\label{mesh:property:2}
		\sum_{j=N/2+2}^N h_j &\leq C\varepsilon\ln N, 
		\end{align}
	\end{subequations}
	where $C>0$ is independent of $\varepsilon$ and $N$.
\end{lemma}

	\begin{proof}
	It is clear that $N^{-1}\leq h_j=2(1-\tau)N^{-1}\leq  2N^{-1}$ for $0\leq j\leq N/2$.
	For $N/2+1\leq j \leq N$, we have 
	\begin{align*}
	\frac{\sigma\varepsilon}{\alpha}N^{-1}\min\varphi^{\prime}
	\leq  h_j
	\leq  \frac{\sigma\varepsilon}{\alpha}N^{-1}\max\varphi^{\prime},
	\end{align*}
	which leads to the conclusion $C \varepsilon N^{-1} \le h_j\le C N^{-1}$
	for each types of layer-adapted meshes in 
	Table \ref{Table:mesh:functions} because $\varepsilon \le N^{-1}$.
	In addition, 
	\eqref{mesh:property:1} holds, see \cite[Lemma 3.2]{Cheng2020}.
	For the B-mesh and $N\geq 2$, one gets 
	\begin{align*}
	\sum_{j=N/2+2}^N h_j=1-x_{N/2+1}
	=\frac{\sigma\varepsilon}{\alpha}
	\varphi\Big(\frac12-\frac{1}{N}\Big)
	= \frac{\sigma\varepsilon}{\alpha} \ln\Big(\frac{N}{2+N\varepsilon-2\varepsilon}\Big)
	\le\frac{\sigma\varepsilon}{\alpha} \ln \frac{N}{2}\leq C\varepsilon\ln N.
	\end{align*}
	This completes the proof of Lemma \ref{Lemma:mesh:size}.
\end{proof}


\subsection{The LDG method}
Let $\Omega _N:=\{I_j =[x_{j-1},x_j],j = 1,2, \ldots ,N\} $
be a partition of $\Omega $ with $x_0 = 0$ and $x_N = 1$.
Let $h_j = x_j - x_{j - 1}$ be the mesh size of the
element $I_j$, $j = 1,2, \ldots ,N$.
The discontinuous finite element space is defined as
\begin{equation}\label{discontinuous finite element space}
\mathcal{V}_N = \{ v \in L^2(\Omega ):
v|_{I_j} \in \mathcal {P}^k(I_j),j = 1,\ldots,N\} ,
\end{equation}
where $\mathcal {P}^k(I_j)$ denotes the space
of polynomials in $I_j$ of degree at most $k \ge 0$.
The functions in $\mathcal {V}_N$ allow discontinuity across element interfaces.
We denote $v_j^ \pm  = \lim _{x \to x_j^ \pm }v(x)$.
We define a jump as $\jump{v}_j=v_j^+-v_j^-$
for $j = 1,2, \ldots ,N - 1$, $\jump{v}_0=v_0^+$
and $\jump{v}_N=-v_N^-$.

Rewrite the problem (\ref{SPP:3RD})
	into an equivalent first-order system
\begin{align*}
&p = u^{\prime},\quad
q = \varepsilon p^{\prime},\quad
q^{\prime} - (ap)^{\prime} + bu^{\prime} + cu = f.
\end{align*}
Then, the LDG method reads:\\
	\indent
	Find the numerical solution
$\wph=(\uph,\pph,\qph) \in \mathcal {V}_N^3 :=
\mathcal {V}_N \times \mathcal {V}_N \times \mathcal{ V}_N$
such that in each element $I_j$,
\begin{subequations}\label{LDG:SCHEME}
	\begin{align}
	\label{LDG form p}
	\dual{\pph}{r}_{I_j}+\dual{\uph}{r^{\prime}}_{I_j}-
	\hat {\uph}_{j}r_{j}^{-}+\hat {\uph}_{j-1}r_{j-1}^{+}&= 0,
	\\
	\label{LDG form q}
	\dual{\qph}{s}_{I_j}+\varepsilon(\dual{\pph}{s^{\prime}}_{I_j}-
	\hat{\pph}_{j}s_{j}^{-}+\hat{\pph}_{j-1}s_{j-1}^{+}) &= 0,
	\\
	-\dual{\qph}{v^{\prime}}_{I_j}+
	\hat{\qph}_{j}v_{j}^{-}-\hat{\qph}_{j-1}v_{j-1}^{+}\nonumber\\
	+\dual{a\pph}{v^{\prime}}_{I_j}-{a_j}\tilde{\pph}_{j}v_{j}^{-}
	+a_{j-1}\tilde{\pph}_{j-1}v_{j-1}^{+}\nonumber\\
	\label{LDG form u}
	-\dual{b\uph}{v^{\prime}}_{I_j}
	+\widetilde{b\uph}_{j}v_{j}^{-}-\widetilde{b\uph}_{j-1}v_{j-1}^{+}
	+\dual{(c - b^{\prime})\uph}{v}_{I_j}&=\dual{f}{v}_{I_j}
	\end{align}
\end{subequations}
hold for any test function $\bm{\chi}=(v,r,s)\in\mathcal{V}_N^3$,
where $\dual{\cdot}{\cdot}_{I_j}$
is the inner product in ${L^2}(I_j)$, 
the "hat" terms and "tilde" terms are
the numerical fluxes defined in Table \ref{table:fluxes}.
\begin{table}[htp]
	\caption{Numerical fluxes in \eqref{LDG:SCHEME}}
	\label{table:fluxes}
	\small
	\centering
	\begin{tabular}{cccccc}
		\hline
		& $j=1,2,\dots,N-1$
		&
		& $j=0$
		&
		& $j=N$
		\\
		\hline
		$\hat{U}_{j}$             & $U^{-}_{j}$ &   & $0$   &&$0$ \\
		$\hat{P}_{j}$             & $P^{+}_{j}$ &   & $P^{+}_{0}$   &&$0$ \\
		$\hat{Q}_{j}$    & $Q^{+}_{j}$  &  & $Q^{+}_{0}$ &&$Q^{-}_{N}$  \\
		$\tilde{P}_{j}$             & $P^{+}_{j}$ &   & $P^{+}_{0}$   &&$P^{-}_{N}$ \\
		\smallskip
		$\widetilde{bU}_{j}$
		& $\frac{b_j + |b_j|}{2}U_j^ -  + \frac{b_j - |b_j|}{2}U_j^ +$ &
		& $ \frac{b_0 - |b_0|}{2}U_0^ +$   &
		&$\frac{b_N + |b_N|}{2}U_N^ -$ \\
		\hline
	\end{tabular}
\end{table}
Note that at the interior element boundary points,
we choose an upwind flux
for the convection part $\widetilde{bU}$,
an alternating flux for the diffusion part $(\hat{U}_{j},\tilde{P}_{j})$,
and the dispersive part $(\hat{U}_{j},\hat{Q}_{j})$.
The choice of flux $\hat{P}_{j}$
ensures stability of the numerical scheme.
This completes the definition of the LDG method for problem (\ref{SPP:3RD}).

Denoting by $\dual{w}{v}=\sum_{j=1}^N\dual{w}{v}_{I_j}$,
we can rewrite the above-mentioned LDG method in a compact form:\\
\indent
Find ${\wph}= (\uph,\pph,\qph)\in\mathcal{V}_N^3$, such that
\begin{equation}\label{compact:form}
B(\wph;\vph)=\dual{f}{v}
\hspace{2em}\forall \vph=(v,r,s)\in\mathcal{V}_N^3,
\end{equation}
where
\begin{align}\label{bilinear form}
B(\wph;\vph)&=\dual{\pph}{r}+\dual{\uph}{r^{\prime}}
+\sum_{j=1}^{N-1}\uph_{j}^{-}\jump{r}_{j}
\nonumber \\
&+\dual{\qph}{s}+\varepsilon\left(\dual{\pph}{s^{\prime}}
+\sum_{j=1}^{N-1}\pph_{j}^{+}\jump{s}_{j}+\pph_{0}^{+}s_{0}^{+}\right)
\nonumber\\
&-\dual{\qph}{v^{\prime}}-\sum_{j=1}^{N-1}\qph_{j}^{+}\jump{v}_{j}
+\qph_{N}^{-}v_{N}^{-}-\qph_{0}^+v_{0}^{+}
\\
&+\dual{a\pph}{v^{\prime}}+\sum_{j=1}^{N-1}a_j\pph_{j}^{+}\jump{v}_{j}
-a_N\pph_{N}^{-}v_{N}^{-}+a_0\pph_{0}^{+}v_{0}^{+}
\nonumber\\
&+\dual{(c - b^{\prime})\uph}{v}-\dual{b\uph}{v^{\prime}}-\sum_{j=1}^{N-1}
\left(\frac{b_j+|b_j|}{2}\uph_{j}^{-}+\frac{b_j-|b_j|}{2}\uph_{j}^{+}\right)\jump{v}_{j}
\nonumber\\
&
+\frac{b_N+|b_N|}{2}\uph_{N}^{-}v_{N}^{-}
-\frac{b_0-|b_0|}{2}\uph_{0}^{+}v_{0}^{+}.
\nonumber
\end{align}

By integrating parts and making trivial manipulations, 
one arrives at the energy norm
\begin{equation}
\begin{split}\label{energy:norm}
\enorm{\wph}^2
&:=
B(\uph,\pph,\qph;\uph,-\qph+a\pph,\pph)\nonumber\\
&
=\frac{\varepsilon}{2}\sum_{j=0}^N\jump{P}_{j}^2+\norm{a^{1/2}P}^2
+\norm{(c-{b^{\prime}/2)}^{1/2}\uph}^2
+\frac{1}{2}\sum_{j=0}^N|b_j|\jump{\uph}_{j}^2,
\end{split}
\end{equation}
which implies the existence and uniqueness of 
numerical solution defined by (\ref{compact:form}),
because $U=P=Q=0$ if $f=0$ and $\vph$ is taken suitably
	in \eqref{compact:form}.

\section{Error analysis}
\label{Sec:error}

	This section focuses on the error estimate in the energy norm \eqref{energy:norm}.
We denote the error by $\bm{e}=(u-\uph,p-\pph,q-\qph)$
and split it into two parts: 
$\bm{e}=\bm{w}-\wph=(\bm{w}-\bm{\Pi w})-(\wph-\bm{\Pi w}):=\bm{\eta-\xi}$ with
\begin{subequations}
	\begin{align}
	\label{eta error}
	\bm{\eta}&=(\eta_u,\eta_p,\eta_q)
	=(u-\pi^{-}u,p-\pi^{+}p,q-\pi^{+}q),\\
	\label{xi error}
	\bm{\xi} & =(\xi_u,\xi_p,\xi_q)
	=(\uph-\pi^{-}u,\pph-\pi^{+}p,\qph-\pi^{+}q) \in \mathcal {V}_N^3,
	\end{align}
\end{subequations}
	where $\bm{\Pi w}:=(\pi^{-}u,\pi^{+}p,\pi^{+}q): (H^1(\Omega_N))^3\to \mathcal {V}_N^3$ denotes the local Gauss-Radau projection that
	will be defined below.
	
	In Section \ref{Sec:projection}, we present an estimation of \emph{ approximation error} $\bm{\eta}$. 
	Then, we utilise its property to derive 
	the bound of the \emph{ projection error} $\bm{\xi}$
	and, hence, the error $\bm{e}$ in Section \ref{Sec:main:result}.

\subsection{The approximation error}
\label{Sec:projection}

To derive the error estimate,
we use local Gauss-Radau projections $\pi^\pm$,
defined as follows.
For any $z\in H^1(\Omega_N )$, we have
$\pi^\pm z\in\mathcal{V}_N$ satisfies
\begin{align}\label{projection}
\dual{\pi^{+}z}{v}_{I_j}&\; =\dual{z}{v}_{I_j}
\quad \forall v\in{\mathcal P}^{k-1}(I_j),\quad
(\pi^+ z)_{j-1}^+ = z_{j-1}^+ ;\nonumber\\
\dual{\pi^{-}z}{v}_{I_j}&\; = \dual{z}{v}_{I_j}
\quad \forall v\in{\mathcal P}^{k-1}(I_j),
\quad (\pi^- z)_{j}^- = z_{j}^- 
\end{align}
on each element $I_j=[x_{j-1},x_j]$, $j=1,2,\ldots,N$.
From \cite{P.G.Ciarlet.1978}, one could
verify the existence and uniqueness of these projections.
Furthermore,
denote $\norm{v}_{I_j}=\norm{v}_{L^2(I_j)}$
and $\norm{v}_{L^\infty(I_j)}$ for the typical $L^2$
and $L^\infty $ norms on $I_j$ respectively; then,
one obtains the following properties
\begin{subequations}\label{projections}
	\begin{align}
	\label{projection pi-}
	\norm{\pi^- z}_{I_j}&\le C[\norm{z}_{I_j}+ h_j^{1/2}|z^{-}_j|],\\
	\label{projection pi+}
	\norm{\pi^+ z}_{I_j}&\le C[\norm{z}_{I_j}+ h_j^{1/2}|z^{+}_{j-1}|],\\
	\label{L endless}
	\norm{\pi^\pm z}_{L^\infty(I_j)}&\le C\norm{z}_{L^\infty(I_j)},\\
	\label{global projection}
	\norm{z-\pi^\pm z}_{L^\ell(I_j)}&\le
	Ch_j^{k+1}\norm{z^{(k+1)}}_{L^\ell(I_j)},\hspace{2em}\ell = 2,\infty, 
	\end{align}
\end{subequations}
where $C > 0$ is independent of the element size $h_j$ and function $z$.

\begin{lemma}\label{Lemma:approximation:property}
	Let $\Omega_N$ be a layer-adapted mesh (\ref{mesh points})
	with $\sigma \ge k+1.5$.
	For the function $u$ satisfying Proposition \ref{Proposition:regularity}, we have
	\begin{subequations}\label{norm sum}
		\begin{align}
		\label{norm u}
		\norm{u-\pi^{-} u}&\le C\Big[\varepsilon
		{(N^{-1}\max|\psi^{\prime}|)}^{k+1}+N^{-(k+1)}\Big],\\
		\label{norm p}
		\norm{p-\pi^{+} p}&\le C\Big[\sqrt\varepsilon
		{(N^{-1}\max|\psi^{\prime}|)}^{k+1}+N^{-(k+1)} \Big],\\
		\label{norm q}
		\norm{q-\pi^{+} q}&\le C\sqrt\varepsilon(N^{-1}\max|\psi^{\prime}|)^{k+1},\\
		\label{local norm p}
		\norm{p-\pi^{+} p}_{L^\infty(\Omega^f)}&\le C(N^{-1}\max|\psi^{\prime}|)^{k+1},\\
		\label{local norm q}
		\norm{q-\pi^{+} q}_{L^\infty(\Omega^f)}&\le C(N^{-1}\max|\psi^{\prime}|)^{k+1},\\
		\label{u jump norm}
		\left(\sum_{j=0}^{N}\jump{u-\pi^-u}^2_j \right)^{1/2}
		&\le C\left[\varepsilon(N^{-1}\max|\psi^{\prime}|)^{k+1/2}+ N^{-(k + 1/2)}\right],
		\\
		\label{p:jump:norm}
		\left(\sum_{j=0}^{N}\jump{p-\pi^+p}^2_j \right)^{1/2}
		&\le C(N^{-1}\max|\psi^{\prime}|)^{k+1/2},	
		\end{align}
	\end{subequations}
	where $p=u^{\prime}$, $q=\varepsilon u^{\prime\prime}$ and 
	$C > 0$ is independent of $\varepsilon$ and $N$.
\end{lemma}

\begin{proof}
We would like to show these inequalities individually.
The main idea is to fully use the stability and approximation properties
\eqref{projections} for the function $u$ satisfying
Proposition \ref{Proposition:regularity}.

(1)We first show (\ref{norm u}).
For notational simplification, we denote 
	$\eta_{u_z}=u_z-\pi^- u_z$ for $u_z\in\{\bar u,u_\varepsilon\}$.
Using (\ref{global projection}) for the regular part $\bar u$ yields:
\begin{equation}\label{estimate u-}
\norm{\eta_{\bar u}}^2 =\sum_{j=1}^N\norm{\eta_{\bar u}}_{I_j}^2
\le C\sum_{j=1}^N h_j^{2(k+1)}\norm{{\bar u}^{(k+1)}}_{I_j}^2
\le CN^{-2(k+1)}
\end{equation}
because $h_j\le CN^{-1} (j=1,2,\ldots,N)$ from Lemma 2.

To bound the approximation error for the layer component $u_\varepsilon$,
we proceed as follows.
First, using (\ref{projection pi-}) and (\ref{u decomposition}), we get
\begin{align}\label{A1}
\sum_{j=1}^{N/2}\norm{\eta_{u_\varepsilon}}_{I_j}^2
&\le C\sum_{j=1}^{N/2}\left[\norm{u_\varepsilon}_{I_j}^2+N^{-1}|u_\varepsilon(x_j)|^2\right]
\nonumber\\
&\le C\int_0^{1-\tau}\varepsilon^2e^{{- 2\alpha(1-x)}/{\varepsilon}}dx
+CN^{-1}\varepsilon^2e^{{- 2\alpha(1-x_{N/2})}/{\varepsilon}}
\nonumber\\
&\le C\varepsilon^2(\varepsilon+N^{-1})N^{-2\sigma},
\end{align}
where we used
$\psi(1/2)\leq N^{-1}$.
Second,
	using $h_j\le CN^{-1} (j=1,2,\ldots,N)$, (\ref{L endless})
	and (\ref{global projection})
	along with $\sigma \ge k+1$, we obtain
\begin{align}\label{A2}
\sum_{j=N/2+1}^N\norm{\eta_{u_\varepsilon}}_{I_j}^2
&\le CN^{-1}\sum_{j=N/2+1}^N
\norm{u_\varepsilon-\pi^{-}{u_\varepsilon}}_{{L^\infty}(I_j)}^2
\nonumber\\
&\le CN^{-1}\sum_{j=N/2+1}^N\min\Big\{\norm{u_\varepsilon}_{{L^\infty}(I_j)}^2,
h_j^{2(k+1)}\norm{u_\varepsilon^{(k+1)}} _{{L^\infty}(I_j)}^2\Big\}
\nonumber\\
&\le CN^{-1}\varepsilon^2
\sum_{j=N/2+1}^N\Big(\min\Big\{\frac{h_j}{\varepsilon},1\Big\}
e^{{-\alpha(1-{x_j})}/{\sigma \varepsilon}}\Big)^{2(k+1)}
\nonumber\\
&\le CN^{-1}\varepsilon^2\max_{N/2+1\le j\le N}\mathcal{G}_j^{2k+1}
\sum_{j=N/2+1}^N\mathcal{G}_j
\nonumber\\
&\le C\varepsilon^2N^{-1}\Big(N^{-1}\max|\psi^{\prime}| \Big)^{2k+1},
\end{align}
where (\ref{inequality G})-(\ref{sum inequality G}) were used.
Consequently, $\norm{\eta_u}$ follows from (\ref{estimate u-})-(\ref{A2}), triangle inequality and $\sigma \ge k+1$.

(2) We next show (\ref{norm p}) and (\ref{norm q}).
Denote $p =\bar p+{p_\varepsilon}
	:=\bar u^{\prime}+u^{\prime}_{\varepsilon}$,
	where $	|\bar p^{(j)}|\le C$ and $|p_\varepsilon^{(j)}|\le
	C\varepsilon^{-j}e^{-\alpha(1-x)/\varepsilon}$
	for $j\leq k+1$.
	Let $\eta_{p_z}=p_z-\pi^+ p_z$ for $p_z\in\{\bar p,p_\varepsilon\}$.


Evidently 
$\norm{\eta_{\bar p}}\le \norm{\eta_{\bar p}}_{L^{\infty}(\Omega_N)}\le CN^{-(k+1)}$.
Projection $\pi^+$ has good stability
for the monotone increasing function 
$x\mapsto e^{-\alpha(1-x)/\varepsilon}$,
which results in layer approximation in the rough region to satisfy
\begin{align}\label{B1}
\sum_{j=1}^{N/2}\norm{\eta_{p_\varepsilon}}_{I_j}^2
&\le C\sum_{j=1}^{N/2}
\left[\norm{p_\varepsilon}_{I_j}^2 + h_j|p_\varepsilon(x_{j-1})|^2\right]
\le
C \sum_{j=1}^{N/2} \left[\norm{e^{{-\alpha (1-x)}/\varepsilon}}_{I_j}^2 + h_je^{{-2\alpha (1-x_{j-1})}/\varepsilon}\right]\nonumber\\
&\le C\int_0^{1-\tau}e^{{-2\alpha(1-x)}/{\varepsilon}} dx
\le C\varepsilon N^{-2\sigma}.
\end{align}
Hence, for $\sigma \ge k+1$ we get 
\begin{align}\label{eta:p:smooth:region}
\norm{\eta_{p}}_{\Omega^c}
\leq C\Big[\norm{\eta_{p_\varepsilon}}_{\Omega^c}
+\norm{\eta_{\bar p}}_{\Omega^c}\Big]
\leq C\left[N^{-(k+1)} + \sqrt \varepsilon N^{-\sigma} \right]
\leq CN^{-(k+1)}.
\end{align}

The proof is similar as in (\ref{A2}), for a larger $\sigma \ge k+1.5$,
one bounds the layer approximation in the fine region by
\begin{align}\label{B2}
\sum_{j=N/2+1}^N\norm{\eta_{p_\varepsilon}}_{I_j}^2
&\le C\sum_{j={N/2+1}}^N\min \Big\{h_j^{2(k+1)}
\norm{{p_\varepsilon}^{(k+1)}}_{I_j}^2,
h_j\big|p_\varepsilon(x_{j-1})\big|^2
+\norm{p_\varepsilon}_{I_j}^2 \Big\}
\nonumber\\
& \le C\sum_{j=N/2+1}^N \min \Big\{
{\Big(\frac{h_j}{\varepsilon}\Big)}^{2(k+1)},1\Big\}
\norm{e^{{-\alpha(1-x)}/{\varepsilon}}}_{I_j}^2
\nonumber\\
&\le C\sum_{j=N/2+1}^N \varepsilon\Big(
\min \Big\{{\frac{h_j}{\varepsilon},1}\Big\}
e^{{-\alpha(1-{x_j})}/{\sigma\varepsilon}}\Big)^{2(k+{3/2})}
\nonumber\\
&\le C\varepsilon \max_{N/2+1\le j \le N}
{\mathcal{G}_j}^{2(k+1)}\sum_{j=N/2+1}^N \mathcal {G}_j
\le C\varepsilon( N^{-1}\max|\psi^{\prime}|)^{2(k+1)}.
\end{align}
Hence, we get 
\begin{align}\label{eta:p:layer:region}
\norm{\eta_{p}}_{\Omega^f}\leq 
C\Big[\norm{\eta_{p_\varepsilon}}_{\Omega^f}
+\norm{\eta_{\bar p}}_{\Omega^f}\Big]
\leq
C\left[\sqrt{\varepsilon}(N^{-1}\max|\psi^{\prime}|)^{(k+1)}
+|\Omega^f|^{1/2}N^{-(k+1)}\right].
\end{align}
Due to \eqref{eta:p:smooth:region} and \eqref{eta:p:layer:region},
so \eqref{norm p} is proved. In a similar fashion, (\ref{norm q}) can be proved.

(3) We now show (\ref{local norm p}) and (\ref{local norm q}).
This follows from
stability and approximation property of $\pi^+$
under $L^{\infty}$norm.
For example, for each element $I_j$,
\[\norm{\eta_{\bar p}}_{{L^\infty}(I_j)}
\le CN^{-(k+1)}\norm{{\bar p}^{(k+1)}}_{{L^\infty}(I_j)}
\le CN^{-(k+1)}.\]
When $j=N/2+1,\ldots,N$, using (\ref{L endless}), (\ref{global projection}) and (\ref{inequality G}),
we get
\begin{align}\label{local:estimate:pep}
\norm{\eta_{p_\varepsilon}}_{{L^\infty}(I_j)}
&\le C\min\Big\{\norm{p_\varepsilon}_{{L^\infty}(I_j)},
h_j^{k+1}\norm{p_\varepsilon^{(k+1)}}_{{L^\infty}(I_j)}\Big\}
\nonumber\\
&\le C{\mathcal{G}_j}^{k+1}\le C(N^{-1}\max|\psi^{\prime}|)^{k+1}.
\end{align}
This immediately leads to \eqref{local norm p}.
Analogously, we can prove (\ref{local norm q}).

(4)We then present (\ref{u jump norm}).
From $L^\infty$approximation property (\ref{global projection}) and (\ref{u decomposition}), 
we have
\[
\sum_{j=0}^{N}\jump{\eta_{\bar u}}_j^2
\le C\sum_{j=1}^N \norm{\eta_{\bar u}}_{{L^\infty}(I_j)}^2
\le C\sum_{j=1}^N N^{-2(k+1)}\le CN^{-(2k+1)}.
\]
Using Lemma 1 and (\ref{L endless}), (\ref{global projection}) with
	$\ell=\infty$, Proposition \ref{Proposition:regularity}
	and $\sigma\geq k+1$, we get

\begin{align}\label{sum:u:jump:norm}
\sum_{j=0}^{N}\jump{\eta_{u_\varepsilon}}_j^2
&\le C\sum_{j=1}^{N/2}
\norm{\eta_{u_\varepsilon}}_{{L^\infty}(I_j)}^2
+ C\sum_{j= N/2+1}^N
\norm{\eta_{u_\varepsilon}}_{{L^\infty}(I_j)}^2
\nonumber\\
&
	\le C \sum_{j=1}^{N/2} \norm{u_\varepsilon}_{{L^\infty}(I_j)}^2
	+C\sum_{j= N/2+1}^N
	\min\Big\{\norm{u_\varepsilon}^2_{{L^\infty}(I_j)},
	h_j^{2(k+1)}\norm{u_\varepsilon^{(k+1)}}^2_{{L^\infty}(I_j)}\Big\}
\nonumber\\
&
\le
	C\varepsilon^2 N^{-2\sigma+1}
	+C\varepsilon^2 \max_{N/2+1 \le j \le N}{\mathcal{G}_j}^{2k+1}
	\sum_{j=N/2+1}^N \mathcal{G}_j
%
\nonumber\\
&\le C\varepsilon^2(N^{-1}\max|\psi^{\prime}|)^{2k+1}.
\end{align}
Consequently, (\ref{u jump norm}) follows.

(5) We finally show (\ref{p:jump:norm}).
The proof is similar as in \eqref{u jump norm}. 
In fact, one has
\begin{align*}
\sum_{j=0}^{N}\jump{\eta_{p}}_j^2
&\le C\sum_{j=1}^N \norm{\eta_{\bar p}}_{{L^\infty}(I_j)}^2
+C\sum_{j=1}^{N}
\norm{\eta_{p_\varepsilon}}_{{L^\infty}(I_j)}^2
\\
& \le CN^{-(2k+1)} + C N^{-2\sigma+1}
+ C\max_{N/2+1 \le j \le N}{\mathcal{G}_j}^{2k+1}
\sum_{j=N/2+1}^N \mathcal{G}_j
\\
& \le C(N^{-1}\max|\psi^{\prime}|)^{2k+1}.
\end{align*}
This completes the whole proof of this lemma.
\end{proof}

\subsection{Main result}
\label{Sec:main:result}
We obtain the following energy-norm error estimate:

\begin{theorem}
	\label{thm:error:estimate}
	Let $\bm w =(u,p,q)=(u,u^{\prime},\varepsilon u^{\prime\prime})$
	be the exact solution of problem (\ref{SPP:3RD}) satisfying
	Proposition \ref{Proposition:regularity}
	and $\wph =(\uph,\pph,\qph)\in{\mathcal V}_N^3$ be the numerical solution of the LDG scheme
	(\ref{compact:form}) on the three layer-adapted meshes
	of Table \ref{Table:mesh:functions}
	with $\sigma\ge k+1.5$. Then, we have
	\begin{equation}\label{main:result}
	\enorm{\bm w-\wph} 
\le 
\begin{cases}
C\left(\sqrt{\varepsilon}(N^{-1}\ln N)^{k}+(N^{-1}\ln N)^{k+1}+N^{-(k+1/2)}\right)
&\text{for S-mesh,}\\
C\left(\sqrt{\varepsilon}N^{-k}(\ln N)^{1/2}
+N^{-(k+1/2)}\right)
&\text{for BS-mesh and B-mesh,}
\end{cases}
	\end{equation}
 where $C > 0$ is a constant independent of $\varepsilon $ and $N$.
\end{theorem}


\begin{proof}
Recalling $\bm\eta = \bm w-\bm\Pi\bm w$ and $\bm\xi = \wph-\bm\Pi\bm w$.
Owing to Galerkin orthogonality $B(\bm w - \wph ;\vph)=0 \quad \forall \vph \in \mathcal {V}_N^3$,
we have the following error equation:
\begin{equation}\label{error_equation}
B(\bm{\xi ;\vph})=B(\bm{\eta ;\vph})
\hspace{2em}\forall \vph=(v,r,s) \in \mathcal {V}_N^3.
\end{equation}
Taking $\vph=(\xi_u,-\xi_q+a\xi_p,\xi_p)$ in
\eqref{error_equation}, one has
\begin{equation}\label{expansion}
B(\bm{\eta ;\vph}) := \sum_{i=1}^{12} {S_i},
\end{equation}
where
\begin{align}
S_1:=&(\eta_p,-\xi_q + a\xi_p);
\hspace{0.5em}
S_2:=(\eta_u,-\xi^{\prime}_q+{(a\xi_p)}^\prime);
\hspace{0.5em}
S_3:=\sum_{j=1}^{N-1}(\eta_u)_j^{-} \jump{-\xi_q+a\xi_p}_j;
\nonumber\\
S_4:=&(\eta_q,\xi_p);
\hspace{3em}
S_5:=\varepsilon(\eta_p,\xi^{\prime}_p);
\hspace{2em}
S_6:=\varepsilon \sum_{j=1}^{N-1}(\eta_p)_j^{+}
\jump{\xi_p}_j+\varepsilon(\eta_p)_0^{+}(\xi_p)_0^{+};
\nonumber\\
S_7:=&-(\eta_q,\xi^{\prime}_u);
\hspace{2em}
S_8:=- \sum_{j=1}^{N-1}(\eta_q)_j^{+}\jump{\xi_u}_j
+(\eta_q)_N^{-}(\xi_u)_N^{-}-(\eta_q)_0^{+}(\xi_u)_0^{+} ;
\nonumber\\
S_9:=&(a\eta_p,\xi^{\prime}_u);
\hspace{2em}
S_{10}:=\sum_{j=1}^{N-1}a_j(\eta_p)_j^{+}
\jump{\xi_u}_j-a_N(\eta _p)_N^{-}(\xi_u)_N^{-}+a_0(\eta_p)_0^{+}(\xi_u)_0^{+} ;
\nonumber\\
S_{11}:=&((c-b^{\prime})\eta_u,\xi_u)-(b\eta_u,\xi^{\prime}_u);\nonumber\\
S_{12}:=&- \sum_{j=1}^{N-1}\left(\frac{b_j+|b_j|}{2}(\eta_u)_j^{-}
+\frac{b_j-|b_j|}{2}(\eta_u)_j^{+}\right)\jump{\xi_u}_j
\nonumber\\
&
+\frac{b_N+|b_N|}{2}(\eta_u)_N^{-}(\xi_u)_N^{-}
-\frac{b_0-|b_0|}{2}(\eta _u)_0^{+}(\xi_u)_0^{+} . \nonumber
\end{align}
From the orthogonality of the approximating polynomials
and the exact collocation of the Gauss-Radau projection,
it is easy to find that
\begin{equation}\label{S2=0}
S_3=S_5=S_6=S_7= 0.
\end{equation}

Using the Cauchy-Schwarz inequality,  
inverse inequality,
the definition \eqref{projection} of the Gauss-Radau projection
and $a\geq \alpha>0$, we obtain
\begin{align}
\label{S2:estimate}
|S_2|&
=|(\eta_u,(a-\bar a)\xi^{\prime}_p)+(\eta_u,a^{\prime}\xi_p)|
\nonumber\\
&\le \sum_{j=1}^N \left(\norm{\eta_u}_{I_j}
\cdot h_j\norm {a^{\prime}}_{{L^\infty}(I_j)} \cdot Ch_j^{-1}\norm{\xi_p}_{I_j}\right)
+ C\norm{\eta_u}\norm{a^{1/2}\xi_p}
\nonumber\\
&\le C\norm{\eta_u}\norm{a^{1/2}\xi_p}
\le C\norm{\eta_u}\enorm{\bm\xi}
\le C\norm{\eta_u}^2 +\frac{1}{28}\enorm{\bm\xi}^2,
\end{align}
where $\bar a $ is a piecewise constant function defined by $\bar a = \frac{1}{h_j}\int_{I_j} a \rm dx $ on each element $I_j$.

Analogously, one has
\begin{align}\label{S4:estimate}
|S_4| & \le \norm{\eta_q}\norm{\xi_p}
\le C\norm{\eta_q} ^2
+ \frac{1}{28}\enorm{\bm\xi}^2,
\\
|S_8| & =|(\eta_q )_N^{-}(\xi_u)_N^{-}|
\le C\norm{\eta_q}_{{L^\infty}(\Omega ^f)}^2
+ \frac{1}{28}\enorm{\bm\xi}^2,
\\
|S_9| &=|((a-\bar a)\eta_p,\xi^{\prime}_u)|
\le
C\norm{\eta_p}\enorm{\bm\xi}\le C\norm{\eta_p}^2
+\frac{1}{28} \enorm{\bm\xi}^2,
\\
\label{S10 estimate}
|S_{10}|&=|-a_N(\eta_p)_N^{-}(\xi_u)_N^{-}|
\le C\norm{\eta_p}_{{L^\infty}(\Omega ^f)}^2
+\frac{1}{28}\enorm{\bm\xi}^2,
\\
|S_{11}| & =|((c-b^{\prime})\eta_u,\xi _u)
+((\bar b-b)\eta_u,\xi^{\prime}_u)|
\nonumber\\
&\le C\left\| {{\eta _u}} \right\| \interleave \bm{\xi} \interleave_E
\le C\norm{\eta_u}^2
+\frac1{28}\enorm{\bm\xi}^2.
\end{align}
To bound $S_{12}$, we note that
$(\eta_u)_j^{-}=0$ for $j=1,2,\dots,N$,
$(\eta_u)_j^{+}=\jump{\eta_u}_j$ for $j=0,1,\dots,N-1$ and
\[
\left|\frac{b_j-|b_j|}{2}(\eta_u)_j^{+}\jump{\xi_u}_j\right|
\leq |\jump{\eta_u}_j|(|b_j||\jump{\xi_u}_j|).
\]
Using the Cauchy-Schwarz inequality, we obtain
\begin{align}
\label{S11:estimate}
|S_{12}| &=
\left|- \sum_{j=0}^{N-1}\frac{b_j-|b_j|}{2}(\eta_u)_j^{+} \jump{\xi_u}_j\right|
\le \sqrt 2\left(\sum_{j=0}^{N-1}\jump{\eta_u}_j^2\right)^{1/2}
\left(\sum_{j=0}^{N-1}\frac12|b_j|\jump{\xi_u}_j^2\right)^{1/2}
\nonumber\\
&\le C\sum_{j=0}^{N}\jump{\eta_u}_j^2+\frac{1}{28}\enorm{\bm\xi}^2.
\end{align}
Finally, we bound $S_1$.
The main challenge is to bound $\norm{\xi_q}$ which is not included in
	the energy norm $\enorm{\bm\xi}$.
	We intend to seek its control by $\norm{\xi_p}$ and 
	$\left(\frac{\varepsilon}{2}\sum_{j=1}^N\jump{\xi_p}_j^2\right)^{1/2}$.
	This depends on the inherent structure of the LDG scheme,
	as we shall describe.
	Taking $v=r=0$ in \eqref{error_equation} and restricting the test function $s$ to the local element $I_j$, we have that
	\begin{equation*}
		\dual{\xi_q}{s}_{I_j}
		+\varepsilon\Big(\dual{\xi_p}{s^{\prime}}_{I_j}
		-(\hat{\xi}_p)_j s_j^{-}+(\hat{\xi}_p)_{j-1}s_{j-1}^{+}\Big)=\dual{\eta_q}{s}_{I_j}
	\end{equation*}
	in each element $I_j$ for any function $s\in \mathcal{P}^k(I_j)$,
	where we use the property of $\eta_p=p-\pi^+ p$.
	Thus, $(\xi_p,\xi_q)\in \mathcal{V}^2_N$ satisfies \eqref{Variation:form:p:q}, with $F_j(s)=(\eta_q,s)_{I_j}$.
	From \eqref{relationship:p:q}, we have
\begin{align}\label{estimate xiq}
	\norm{\xi_q}_{I_j}
	\le C\varepsilon\left(h_j^{-1}\norm{\xi_p}_{I_j}
	+h_j^{-1/2}|\jump{\xi_p}_j|\right)+\norm{\eta_q}_{I_j}
\end{align}
for $j=1,2,\dots,N$.
Using the Cauchy-Schwarz and Young's inequalities, yields
\begin{align*}
	|(\eta_p,-\xi_q)| &
	\le \sum_{j=1}^N\norm{\eta_p}_{I_j}\norm{\xi_q}_{I_j}
	\nonumber\\
	&\le C\sum_{j=1}^N\norm{\eta_p}_{I_j}
	\left(\varepsilon h_j^{-1}\norm{\xi_p}_{I_j}
	+\varepsilon h_j^{-1/2}|\jump{\xi_p}_j|+\norm{\eta_q}_{I_j}\right)
	\nonumber \\
	&\le C\sum_{j=1}^N
	\left(1+\frac{\varepsilon}{h_j}\right)^2\norm{\eta_p}^2_{I_j}
	+ \frac{1}{8}\norm{a^{1/2}\xi _p}^2
	+ \frac{1}{4}\cdot \frac{\varepsilon}{2}
	\sum_{j=1}^N\jump{\xi_p}_j^2+C\norm{\eta_q}^2.
\end{align*}
So, one gets
\begin{equation}
	\begin{split}
		|S_1|\le& 
		|\dual{\eta_p}{-\xi_q}|+C\norm{\eta_p}^2+\frac18\norm{a^{1/2}\xi_p}^2
		\le C\sum_{j=1}^N
		\left(1+\frac{\varepsilon}{h_j}\right)^2\norm{\eta_p}^2_{I_j}
		+C\norm{\eta_q}^2
		+\frac{1}{4}\enorm{\bm\xi}^2.
	\end{split}
\end{equation}
Collecting the above estimates, 
we get
\begin{align}\label{xi:last:inequality}
\enorm{\bm\xi}^2\le 
&C\left[\sum_{j=1}^N
\left(1+\frac{\varepsilon}{h_j}\right)^2\norm{\eta_p}^2_{I_j}
+ \norm{\eta_u} ^2 
+ \norm{\eta_q} ^2 
+ \sum_{j=0}^{N}\jump{\eta_u}_j^2
+\norm{\eta_p}_{{L^\infty}(\Omega ^f)}^2
+\norm{\eta_q}_{{L^\infty}(\Omega ^f)}^2
\right].
\end{align}
In the sequel, we shall estimate the first term on the right hand side of \eqref{xi:last:inequality} for each type of layer-adapted meshes.
For S-mesh,
one uses \eqref{eta:p:smooth:region} and \eqref{eta:p:layer:region} to get that
\begin{align*}
\sum_{j=1}^N\left(1+\frac{\varepsilon}{h_j}\right)^2\norm{\eta_p}^2_{I_j}
\leq &\; C(1+\varepsilon N)^2\norm{\eta_p}^2_{\Omega^c}
+C(1+N (\ln N)^{-1})^2\norm{\eta_p}^2_{\Omega^f}
\\
\leq &\; C (1+\varepsilon N)^2 N^{-2(k+1)}
+C\left(N (\ln N)^{-1}\right)^2\left[\varepsilon(N^{-1}\ln N)^{2(k+1)}
+\varepsilon N^{-2(k+1)}\ln N\right]
\\
\leq &\; C \varepsilon (N^{-1}\ln N)^{2k}+ CN^{-2(k+1)}.
\end{align*}
For BS-mesh, using Lemma \ref{Lemma:mesh:size}, \eqref{eta:p:smooth:region} and \eqref{eta:p:layer:region}, we have
\begin{align*}
\sum_{j=1}^N\left(1+\frac{\varepsilon}{h_j}\right)^2\norm{\eta_p}^2_{I_j}
\leq &\; C(1+\varepsilon N)^2\norm{\eta_p}^2_{\Omega^c}
+C N^2\norm{\eta_p}^2_{\Omega^f}
\leq  C \varepsilon N^{-2k}\ln N+ CN^{-2(k+1)}.
\end{align*}
For B-mesh, we obtain from Lemma \ref{Lemma:mesh:size}, \eqref{eta:p:smooth:region} and \eqref{local norm p} that
\begin{align*}
\sum_{j=1}^N\left(1+\frac{\varepsilon}{h_j}\right)^2\norm{\eta_p}^2_{I_j}
\leq &\; C(1+\varepsilon N)^2\norm{\eta_p}^2_{\Omega^c}
+\sum_{j=N/2+2}^N
\left(1+\frac{\varepsilon}{h_j}\right)^2\norm{\eta_p}^2_{I_j}
+\left(1+\frac{\varepsilon}{h_{N/2+1}}\right)^2\norm{\eta_p}^2_{I_{N/2+1}}
\\
\leq &\; C(1+\varepsilon N)^2\norm{\eta_p}^2_{\Omega^c}
+C N^2 \left(\sum_{j=N/2+2}^N h_j\right) \norm{\eta_p}^2_{L^{\infty}(\Omega^f)}
+C \norm{\eta_p}^2
\\
\leq &\; C \varepsilon N^{-2k}\ln N+ CN^{-2(k+1)}.
\end{align*}
Inserting the above estimates into \eqref{xi:last:inequality}
and using Lemma \ref{Lemma:approximation:property},
we have
	\begin{align}\label{xi2}
	\enorm{\bm\xi}^2\le 
	\begin{cases}
	C\left(\varepsilon(N^{-1}\ln N)^{2k}+(N^{-1}\ln N)^{2(k+1)}+N^{-(2k+1)}\right)
    &\text{for S-mesh,}\\
    C\left(\varepsilon N^{-2k}\ln N+N^{-(2k+1)}\right)
    &\text{for BS-mesh and B-mesh.}
    \\
	\end{cases}
	\end{align}
Theorem \ref{thm:error:estimate} follows from (\ref{xi2}), Lemma \ref{Lemma:approximation:property} and triangle inequality.
\end{proof}


\begin{remark}
	If $\varepsilon\le N^{-1}$,
	error estimate \eqref{main:result}
	is optimal up to a logarithmic factor
	and uniform with respect to the small parameter $\varepsilon$.
	However, the convergence rate of the $L^2$-error
	$\norm{u-\uph}$ and $\norm{p-\pph}$,
	implied by \eqref{main:result}
	appear to be inferior to the numerical results.	
\end{remark}

\begin{remark}\label{rmk:thm}
If we employ Gauss-Labotto projection for $p$ and $q$ in the last element $I_N$, we don't need to deal with the terms $S_8$ and $S_{10}$.
That means, for $S$-mesh, we have
$\enorm{\bm w-\wph} \leq C\left(\sqrt{\varepsilon}(N^{-1}\ln N)^{k}+N^{-(k+1/2)}\right)$.
Therefore the final error estimate is of form
$O(\sqrt{\varepsilon N} (N^{-1}\max|\psi^{\prime}|)^{k+1/2}+N^{-(k+1/2)})$. 
\end{remark}

\section{Numerical experiments}
\label{Sec:experiments}

In this section, we present numerical results 
to confirm Theorem \ref{thm:error:estimate}.
We consider problem \eqref{SPP:3RD} with $a=b=c=1$.
Assume that $f$ is suitably chosen such that the exact solution is
\begin{align}\label{exact solution}
u(x)=&-\varepsilon e^{- 1/\varepsilon}
+(1-2\varepsilon+2\varepsilon e^{-1/\varepsilon})\sin(\pi x/2)
+\varepsilon e^{-(1-x)/\varepsilon}
\nonumber\\
&+x(1-x)+(\varepsilon-\varepsilon e^{-1/\varepsilon}-1)\sin^2(\pi x/2).
\end{align}

Figures \ref{fig:energy:u}–\ref{fig:energy:p} show $\uph$ and $\pph$
computed by the LDG method on Shishkin mesh, where $k=1$, $N=64$ and $\varepsilon=10^{-2},10^{-4}$.

The LDG method with piecewise polynomials of degree $k = 0,1,2,3$
is carried out on the three-layer-adapted meshes listed
in Table \ref{Table:mesh:functions}, where $\sigma=k+1.5$.
We calculated the convergence rates using the following formulae:
\[
r_2=\frac{\log{e^N}-\log{e^{2N}}}{\log 2},
\quad
r_s=\frac{\log{e^N}-\log{e^{2N}}}{\log(2\ln N/\ln 2N)},
\]
where $e^N$ denotes the error in the $N$-element,
$r_2$ is the convergence rate of the BS-mesh and B-mesh, and
$r_s$ is the convergence rate of the S-mesh with respect to the power $\ln N$.

In Table \ref{Table:energy:error:-4}, we list the energy-error as well as the convergence rate
for  $\varepsilon=10^{-4}$.
One sees that the energy-error converges at a rate of $O((N^{-1}\max|\psi^{\prime}|)^{k+1/2})$.
In Tables \ref{Table:energy:error:-8}-\ref{Table:energy:error:-12}, 
we compute energy-error for  $\varepsilon=10^{-8},10^{-12}$
and find that the data for three types of layer-adapted meshes
is almost the same, the convergence rate is  $O(N^{-(k+1/2)})$.
In Figure \ref{fig:error:energy}, we display 
the plots of the convergence rate for $\varepsilon=10^{-4},10^{-8}$.
These numerical results imply 
that the energy error converges at a rate of 
$O(\varepsilon^{\kappa} (N^{-1}\max|\psi^{\prime}|)^{k+1/2}+N^{-(k+1/2)})$
for some constant $\kappa>0$, which is slightly better
than the predictions in Theorem \ref{thm:error:estimate}
and Remark \ref{rmk:thm}.

Lastly, we point out that the $L^2$-error $\norm{u-\uph}$ and $\norm{p-\pph}$ converges at a rate of $O(\varepsilon^{\kappa} (N^{-1}\max|\psi^{\prime}|)^{k+1}+N^{-(k+1)})$, see Figures \ref{fig:error:u}--\ref{fig:error:p}.

\begin{figure}[htp]
	\centering
	\includegraphics[width=1.8in,height=1.8in]{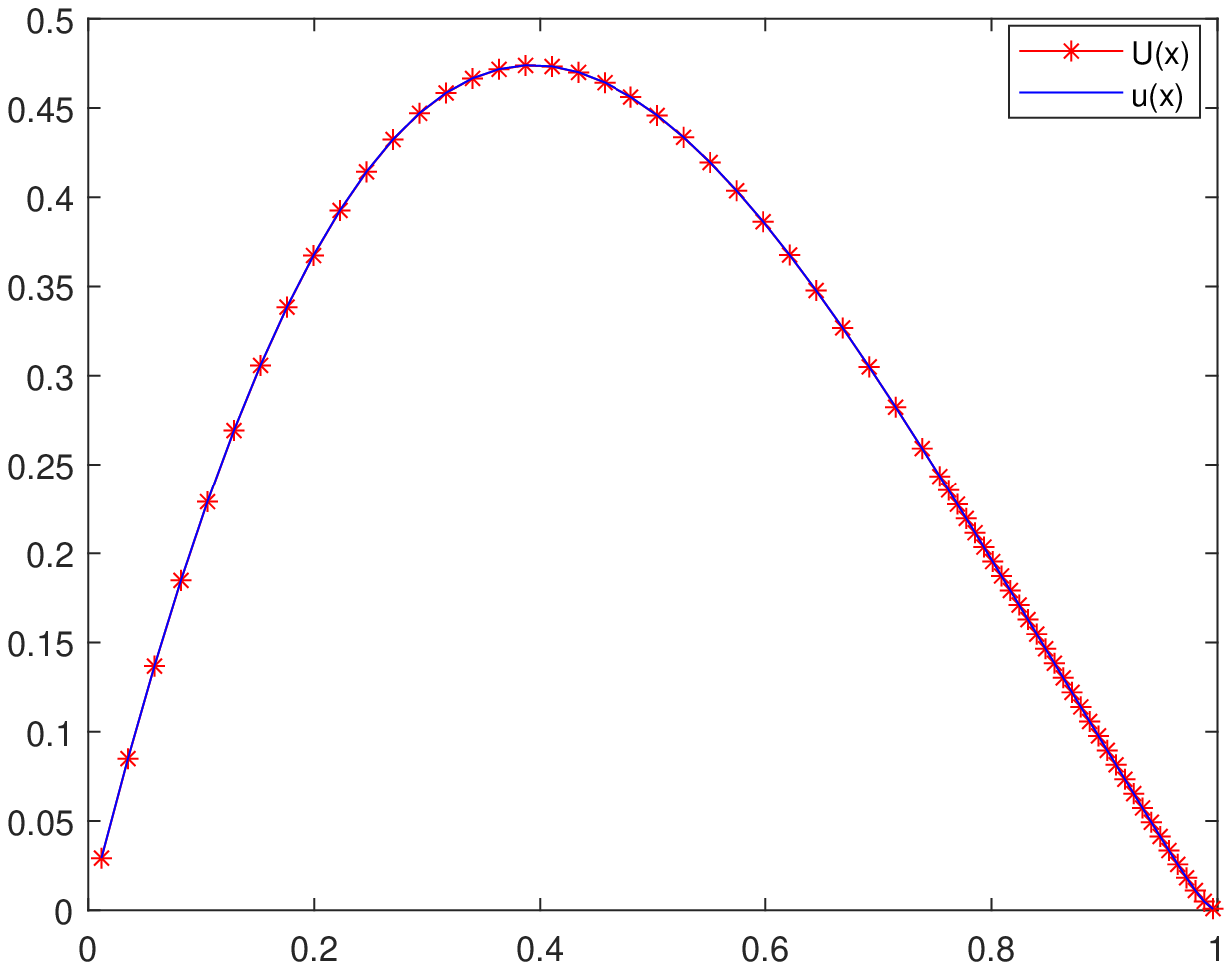}\quad
	\includegraphics[width=1.8in,height=1.8in]{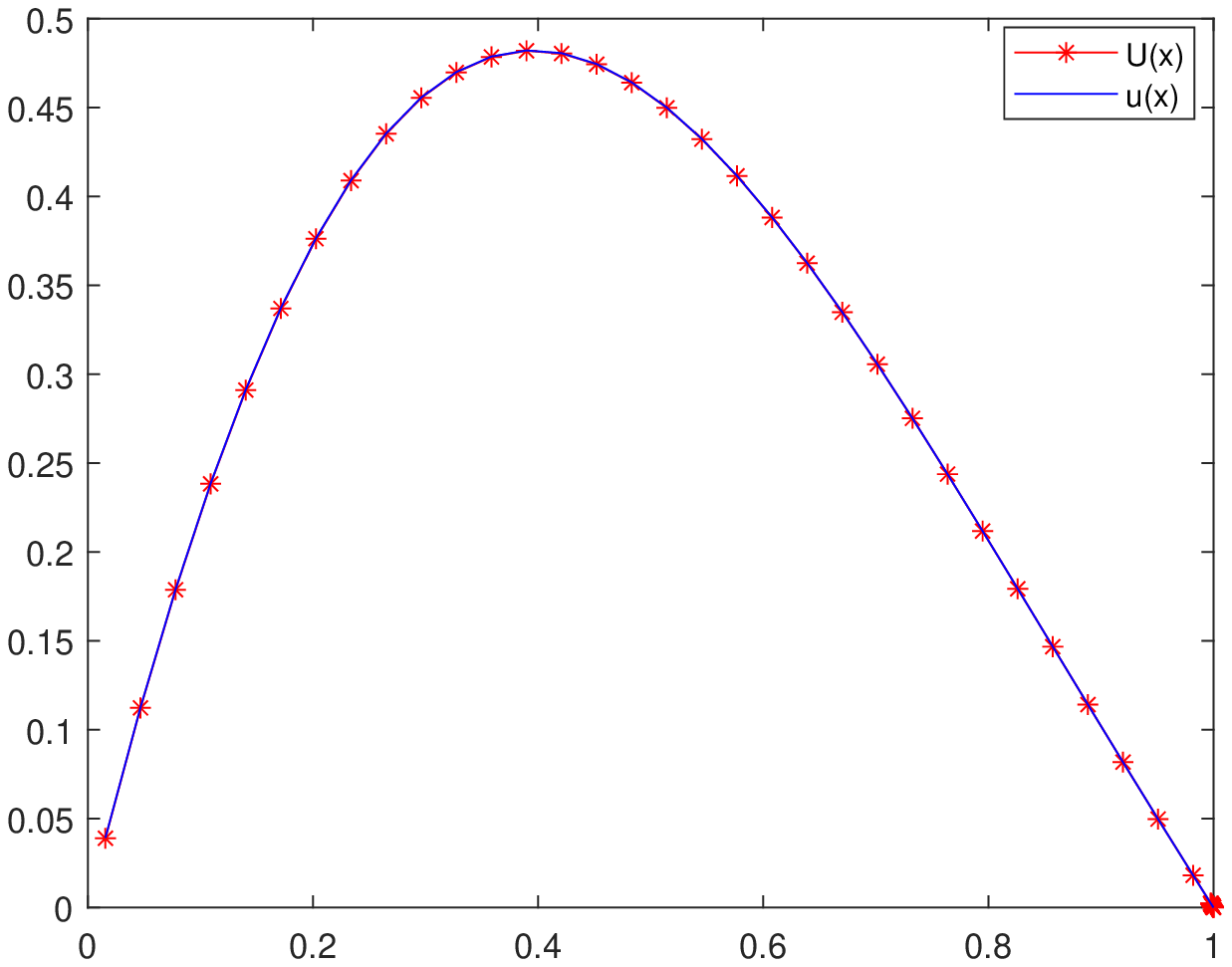}
	\caption{\small
		Numerical approximation of $u$: $\varepsilon=10^{-2}$ (left);
		$\varepsilon=10^{-4}$ (right).
	}
	\label{fig:energy:u}
\end{figure}
\begin{figure}[htp]
	\centering
	\includegraphics[width=1.8in,height=1.8in]{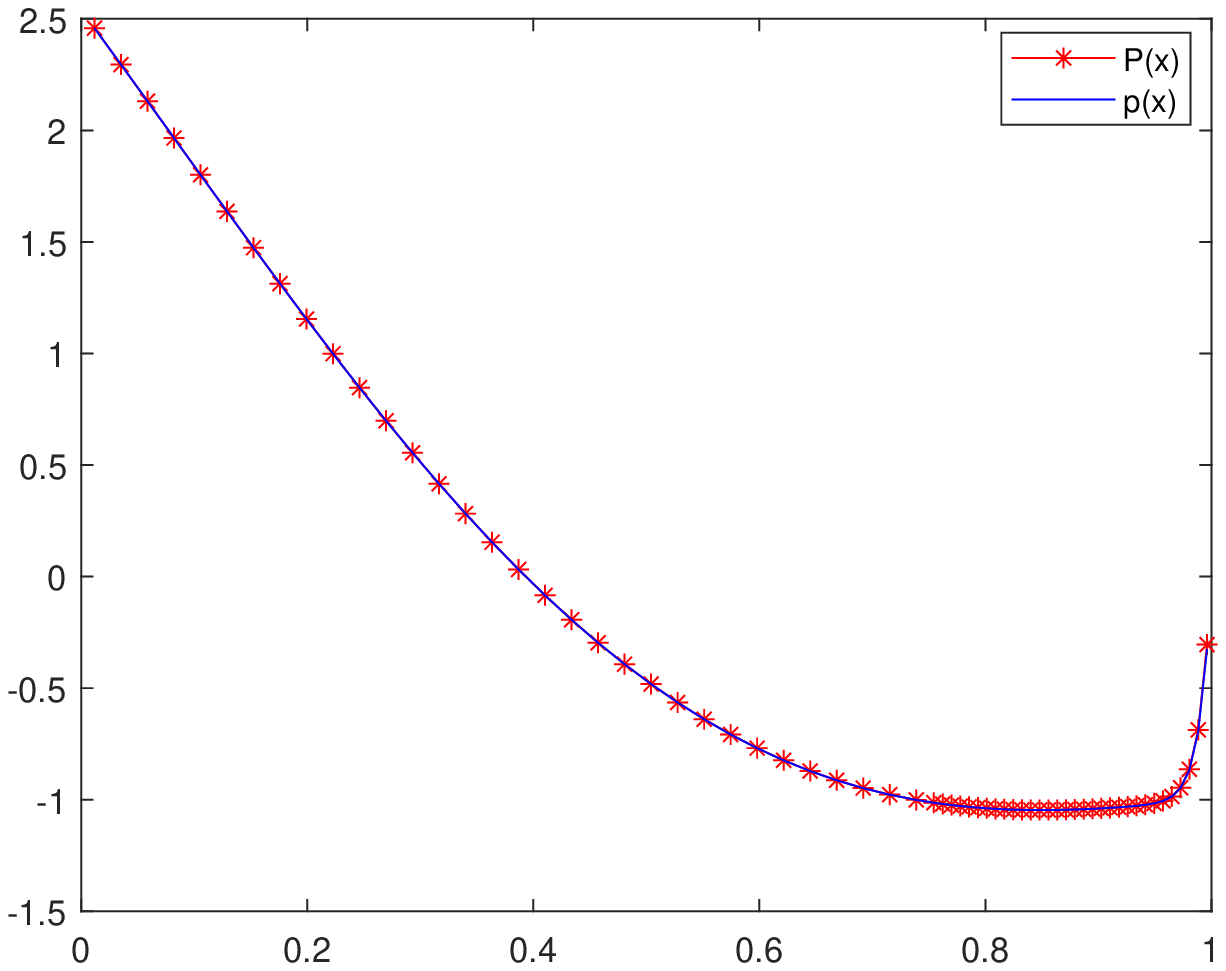}\quad
	\includegraphics[width=1.8in,height=1.8in]{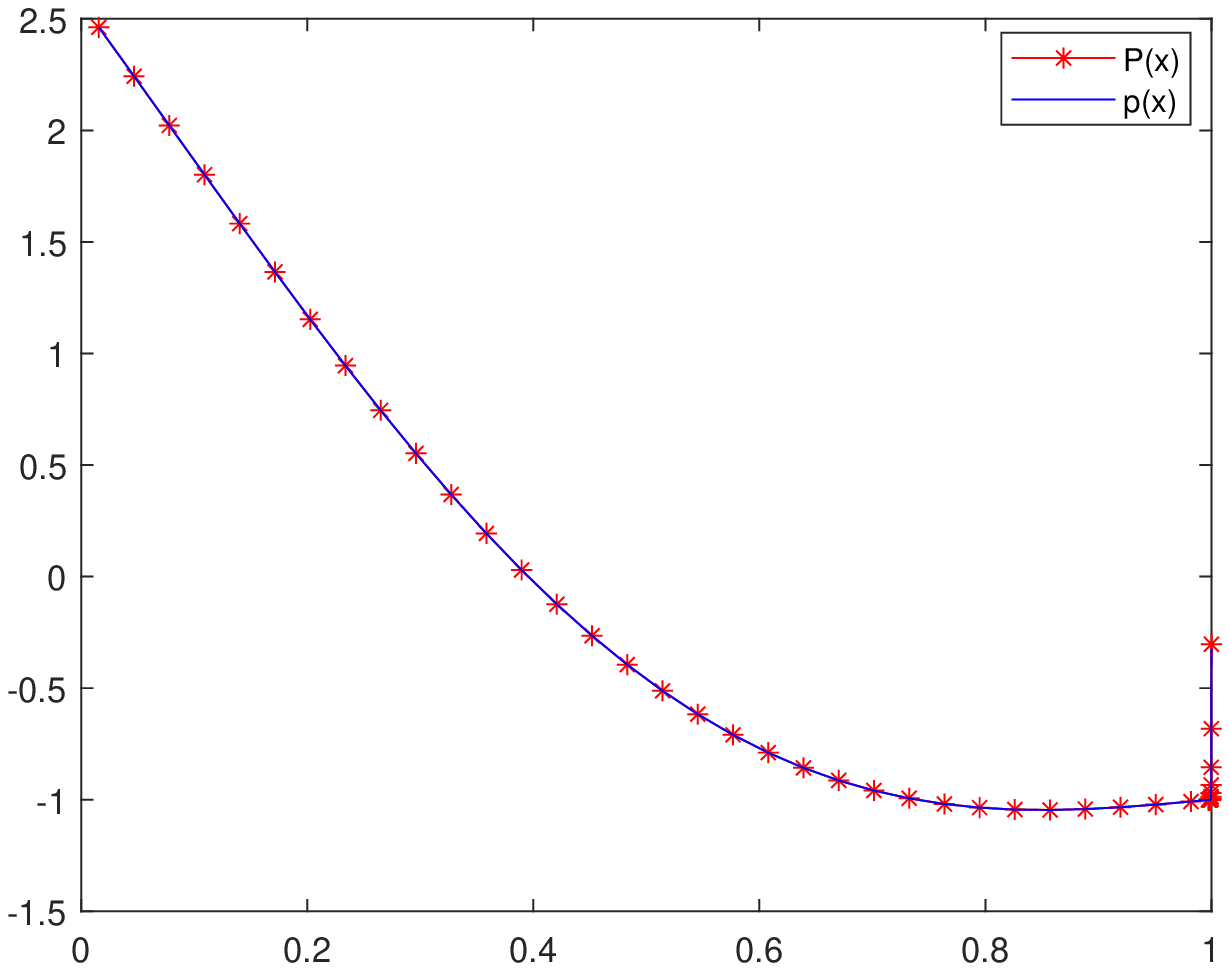}
	\caption{\small
		Numerical approximation of $p$:  $\varepsilon=10^{-2}$ (left);
		$\varepsilon=10^{-4}$ (right).
	}
	\label{fig:energy:p}
\end{figure}

\begin{table}[!ht]
	\centering
	\caption{ Energy-error and convergence rate for $\varepsilon=10^{ -4}$.}
	\label{Table:energy:error:-4}
	\begin{tabular}{ccccccccc}
				\hline
		&&\multicolumn{2}{c}{S-mesh}&\multicolumn{2}{c}{BS-mesh}&\multicolumn{2}{c}{B-mesh}\\ 
		\hline
		&$N$&error&$r_s$-order&error&$r_2$-order&error&$r_2$-order\\
		\hline	
		${{\cal P}^0}$&16&3.87e-01&-&3.87e-01&-&3.87e-01&-\\
		&32&2.40e-01&1.02&2.40e-01&0.69&2.40e-01&0.69\\
	    &64&1.56e-01&0.84&1.56e-01&0.62&1.56e-01&0.62\\
		&128&1.05e-01&0.73&1.05e-01&0.57&1.05e-01&0.57\\
		&256&7.21e-02&0.67&7.21e-02&0.54&7.20e-02&0.54\\
		&512&5.02e-02&0.63&5.02e-02&0.52&5.02e-02&0.52\\
		${{\cal P}^1}$&16&2.57e-02&-&2.57e-02&-&2.56e-02&-\\
		&32&8.71e-03&2.30&8.71e-03&1.56&8.69e-03&1.56\\
		&64&3.01e-03&2.08&3.01e-03&1.53&3.00e-03&1.53\\
		&128&1.05e-03&1.95&1.05e-03&1.52&1.05e-03&1.51\\
		&256&3.70e-04&1.86&3.69e-04&1.51&3.68e-04&1.51\\
		&512&1.30e-04&1.82&1.30e-04&1.51&1.30e-04&1.50\\
		${{\cal P}^2}$&16&6.62e-04&-&6.51e-04&-&6.47e-04&-\\
		&32&1.15e-04&3.72&1.08e-04&2.59&1.08e-04&2.58\\
		&64&2.17e-05&3.26&1.85e-05&2.55&1.84e-05&2.55\\
		&128&4.39e-06&2.96&3.20e-06&2.53&3.19e-06&2.53\\
		&256&9.31e-07&2.77&5.60e-07&2.51&5.58e-07&2.52\\
		&512&2.02e-07&2.66&9.85e-08&2.51&9.82e-08&2.51\\
		${{\cal P}^3}$&16&3.16e-05&-&2.06e-05&-&2.04e-05&-\\
		&32&5.37e-06&3.77&1.75e-06&3.56&1.73e-06&3.56\\
		&64&8.99e-07&3.50&1.52e-07&3.53&1.50e-07&3.53\\
		&128&1.37e-07&3.49&1.33e-08&3.51&1.31e-08&3.52\\
		&256&1.94e-08&3.49&1.16e-09&3.52&1.16e-09&3.50\\
		&512&2.60e-09&3.49&1.02e-10&3.51&1.02e-10&3.51\\ 
		\hline
	\end{tabular}
\end{table}

\begin{table}[!ht]
	\centering
	\caption{ Energy-error and convergence rate for $\varepsilon=10^{-8}$.}
	\label{Table:energy:error:-8}
	\begin{tabular}{ccccccccc}
				\hline
		&&\multicolumn{2}{c}{S-mesh}&\multicolumn{2}{c}{BS-mesh}&\multicolumn{2}{c}{B-mesh}\\ 
		\hline
		&$N$&error&$r_2$-order&error&$r_2$-order&error&$r_2$-order\\
		\hline
		${{\cal P}^0}$&16&3.87e-01&-&3.87e-01&-&3.87e-01&-\\
		&32&2.40e-01&0.69&2.40e-01&0.69&2.40e-01&0.69\\
		&64&1.56e-01&0.62&1.56e-01&0.62&1.56e-01&0.62\\
		&128&1.05e-01&0.57&1.05e-01&0.57&1.05e-01&0.57\\
		&256&7.21e-02&0.54&7.21e-02&0.54&7.20e-02&0.54\\
		&512&5.02e-02&0.52&5.02e-02&0.52&5.02e-02&0.52\\
		${{\cal P}^1}$&16&2.57e-02&-&2.57e-02&-&2.57e-02&-\\
		&32&8.72e-03&1.56&8.72e-03&1.56&8.72e-03&1.56\\
		&64&3.01e-03&1.53&3.01e-03&1.53&3.01e-03&1.53\\
		&128&1.05e-03&1.52&1.05e-03&1.52&1.05e-03&1.52\\
		&256&3.69e-04&1.51&3.69e-04&1.51&3.69e-04&1.51\\
		&512&1.30e-04&1.51&1.30e-04&1.51&1.30e-04&1.51\\	
		${{\cal P}^2}$&16&6.52e-04&-&6.52e-04&-&6.52e-04&-\\
		&32&1.09e-04&2.58&1.09e-04&2.58&1.08e-04&2.59\\
		&64&1.85e-05&2.56&1.85e-05&2.56&1.85e-05&2.55\\
		&128&3.21e-06&2.53&3.21e-06&2.53&3.21e-06&2.53\\
		&256&5.62e-07&2.51&5.62e-07&2.51&5.62e-07&2.51\\
		&512&9.89e-08&2.51&9.89e-08&2.51&9.89e-08&2.51\\
		${{\cal P}^3}$&16&2.06e-05&-&2.06e-05&-&2.06e-05&-\\
		&32&1.76e-06&3.55&1.76e-06&3.55&1.76e-06&3.55\\
		&64&1.53e-07&3.52&1.52e-07&3.53&1.52e-07&3.53\\
		&128&1.34e-08&3.51&1.33e-08&3.51&1.33e-08&3.51\\
		&256&1.19e-09&3.49&1.17e-09&3.51&1.17e-09&3.51\\
		&512&1.07e-10&3.48&1.03e-10&3.51&1.03e-10&3.51\\ 
		\hline
	\end{tabular}
\end{table}

\begin{table}[!ht]
	\centering
	\caption{ Energy-error and convergence rate for $\varepsilon=10^{-12}$.}
	\label{Table:energy:error:-12}
	\begin{tabular}{ccccccccc}
				\hline
		&&\multicolumn{2}{c}{S-mesh}&\multicolumn{2}{c}{BS-mesh}&\multicolumn{2}{c}{B-mesh}\\ 
		\hline
		&$N$&error&$r_2$-order&error&$r_2$-order&error&$r_2$-order\\
		\hline
		${{\cal P}^0}$&16&3.87e-01&-&3.87e-01&-&3.87e-01&-\\
		&32&2.40e-01&0.69&2.40e-01&0.69&2.40e-01&0.69\\
		&64&1.56e-01&0.62&1.56e-01&0.62&1.56e-01&0.62\\
		&128&1.05e-01&0.57&1.05e-01&0.57&1.05e-01&0.57\\
		&256&7.21e-02&0.54&7.21e-02&0.54&7.20e-02&0.54\\
		&512&5.02e-02&0.52&5.02e-02&0.52&5.02e-02&0.52\\
		${{\cal P}^1}$&16&2.57e-02&-&2.57e-02&-&2.57e-02&-\\
		&32&8.72e-03&1.56&8.72e-03&1.56&8.72e-03&1.56\\
		&64&3.01e-03&1.53&3.01e-03&1.53&3.01e-03&1.53\\
		&128&1.05e-03&1.52&1.05e-03&1.52&1.05e-03&1.52\\
		&256&3.69e-04&1.51&3.69e-04&1.51&3.69e-04&1.51\\
		&512&1.30e-04&1.51&1.30e-04&1.51&1.30e-04&1.51\\	
		${{\cal P}^2}$&16&6.52e-04&-&6.52e-04&-&6.52e-04&-\\
		&32&1.09e-04&2.58&1.09e-04&2.58&1.08e-04&2.59\\
		&64&1.85e-05&2.56&1.85e-05&2.56&1.85e-05&2.55\\
		&128&3.21e-06&2.53&3.21e-06&2.53&3.21e-06&2.53\\
		&256&5.62e-07&2.51&5.62e-07&2.51&5.62e-07&2.51\\
		&512&9.89e-08&2.51&9.89e-08&2.51&9.89e-08&2.51\\
		${{\cal P}^3}$&16&2.06e-05&-&2.06e-05&-&2.06e-05&-\\
		&32&1.76e-06&3.55&1.76e-06&3.55&1.76e-06&3.55\\
		&64&1.53e-07&3.52&1.52e-07&3.53&1.52e-07&3.53\\
		&128&1.34e-08&3.51&1.33e-08&3.51&1.33e-08&3.51\\
		&256&1.19e-09&3.49&1.17e-09&3.51&1.17e-09&3.51\\
		&512&1.07e-10&3.48&1.03e-10&3.51&1.03e-10&3.51\\ 
		\hline
	\end{tabular}
\end{table}

\begin{figure}[htp]
	\centering
	\includegraphics[width=1.8in,height=1.8in]{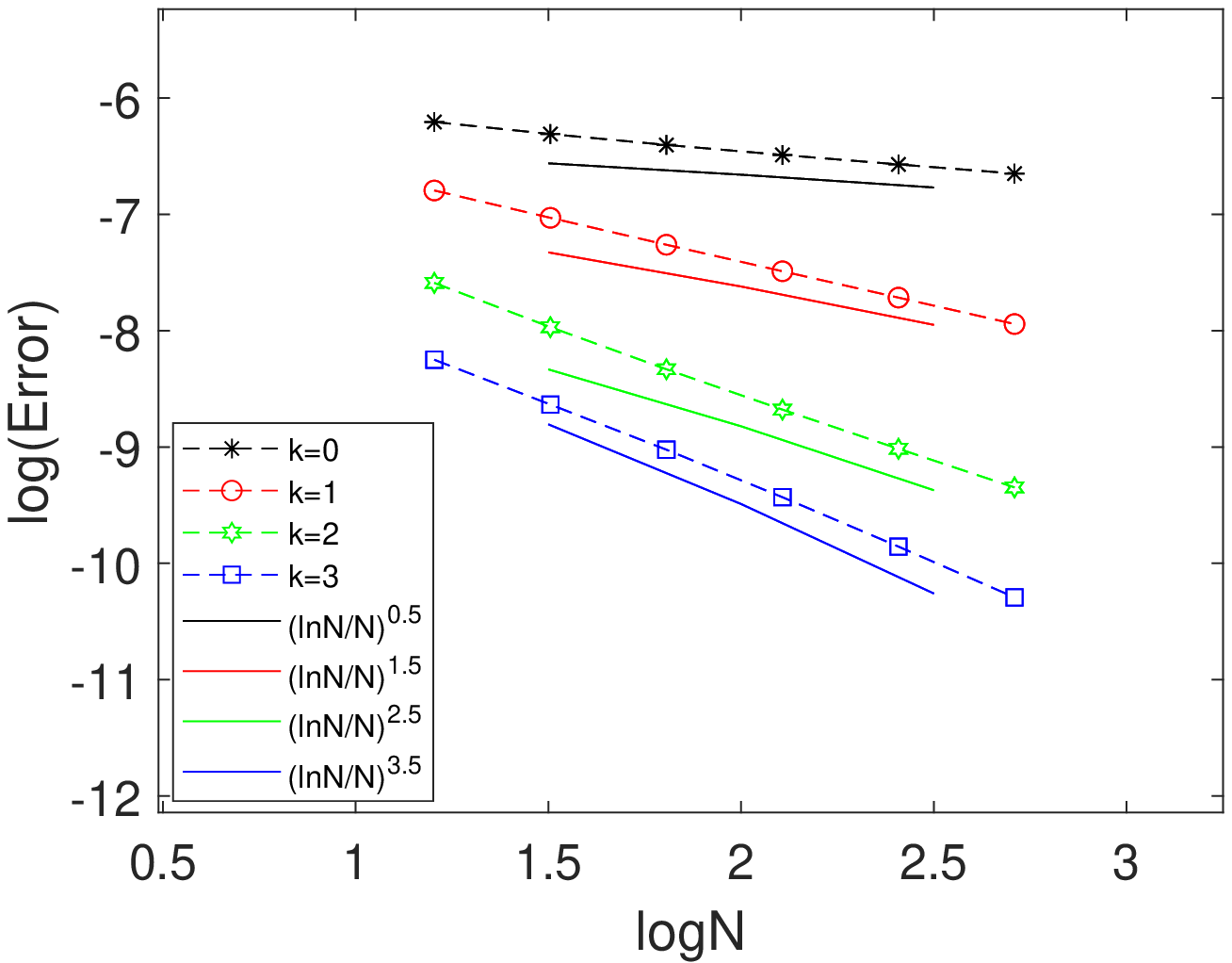}\quad
	\includegraphics[width=1.8in,height=1.8in]{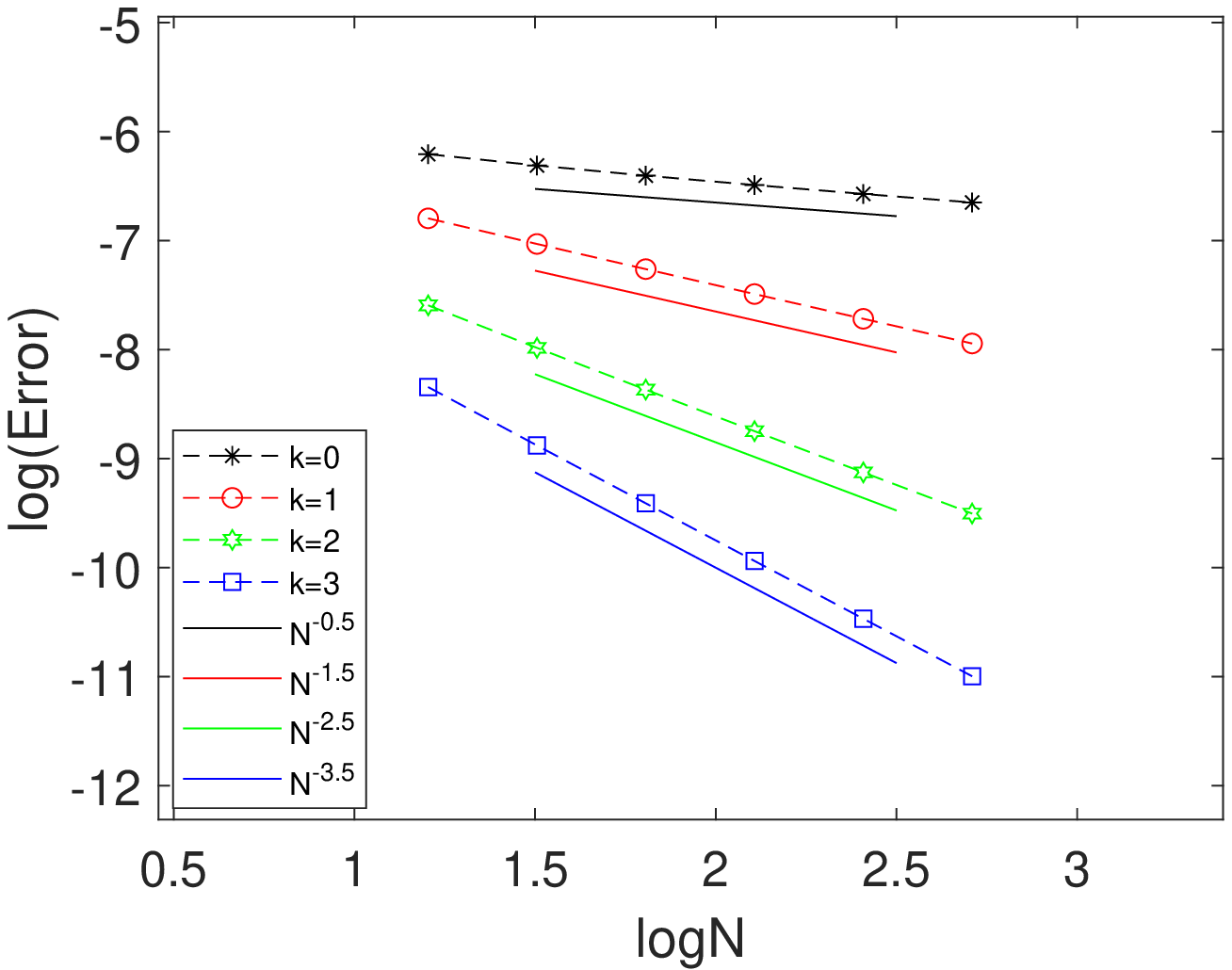}
	\quad
	\includegraphics[width=1.8in,height=1.8in]{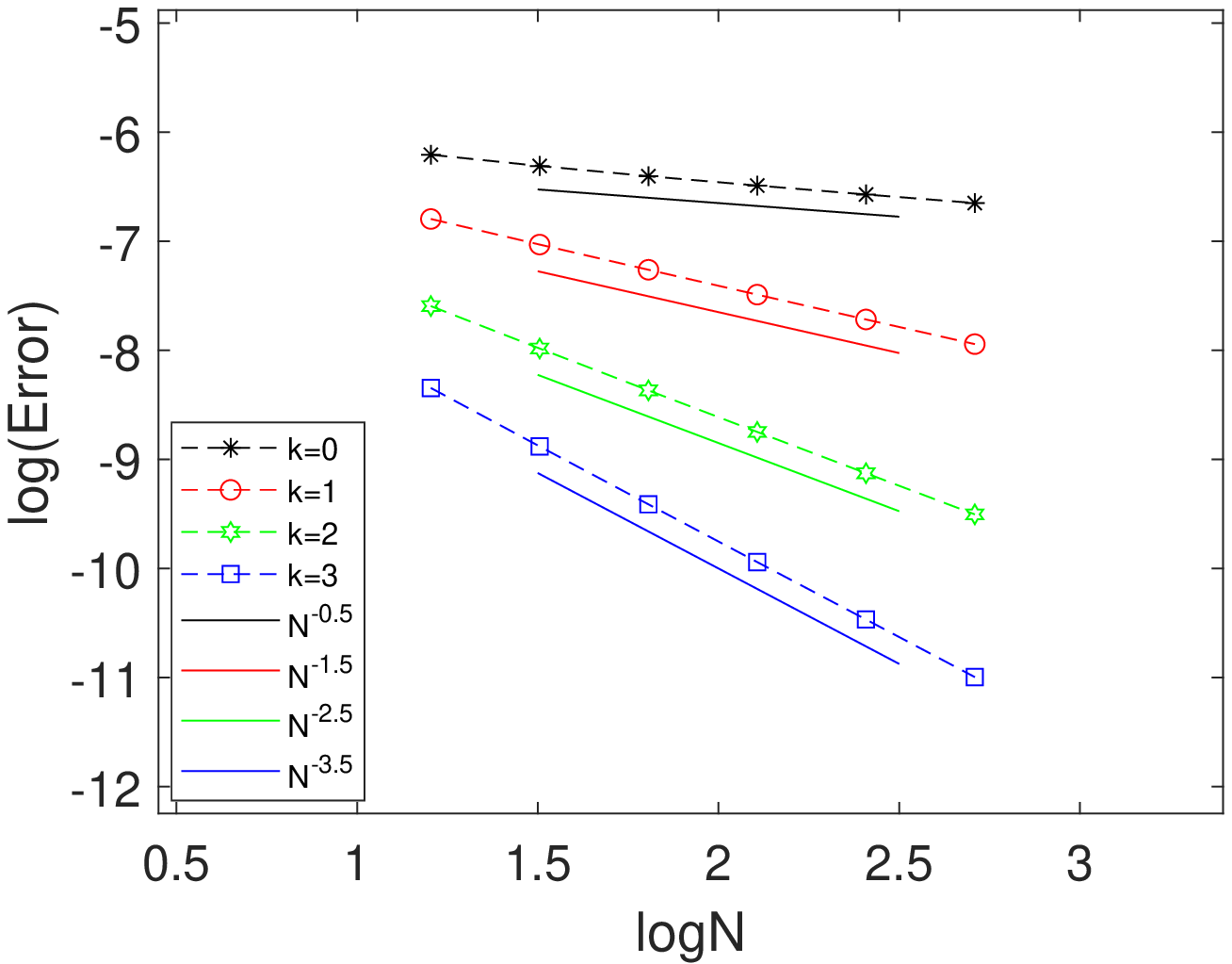}
	\centering
	\includegraphics[width=1.8in,height=1.8in]{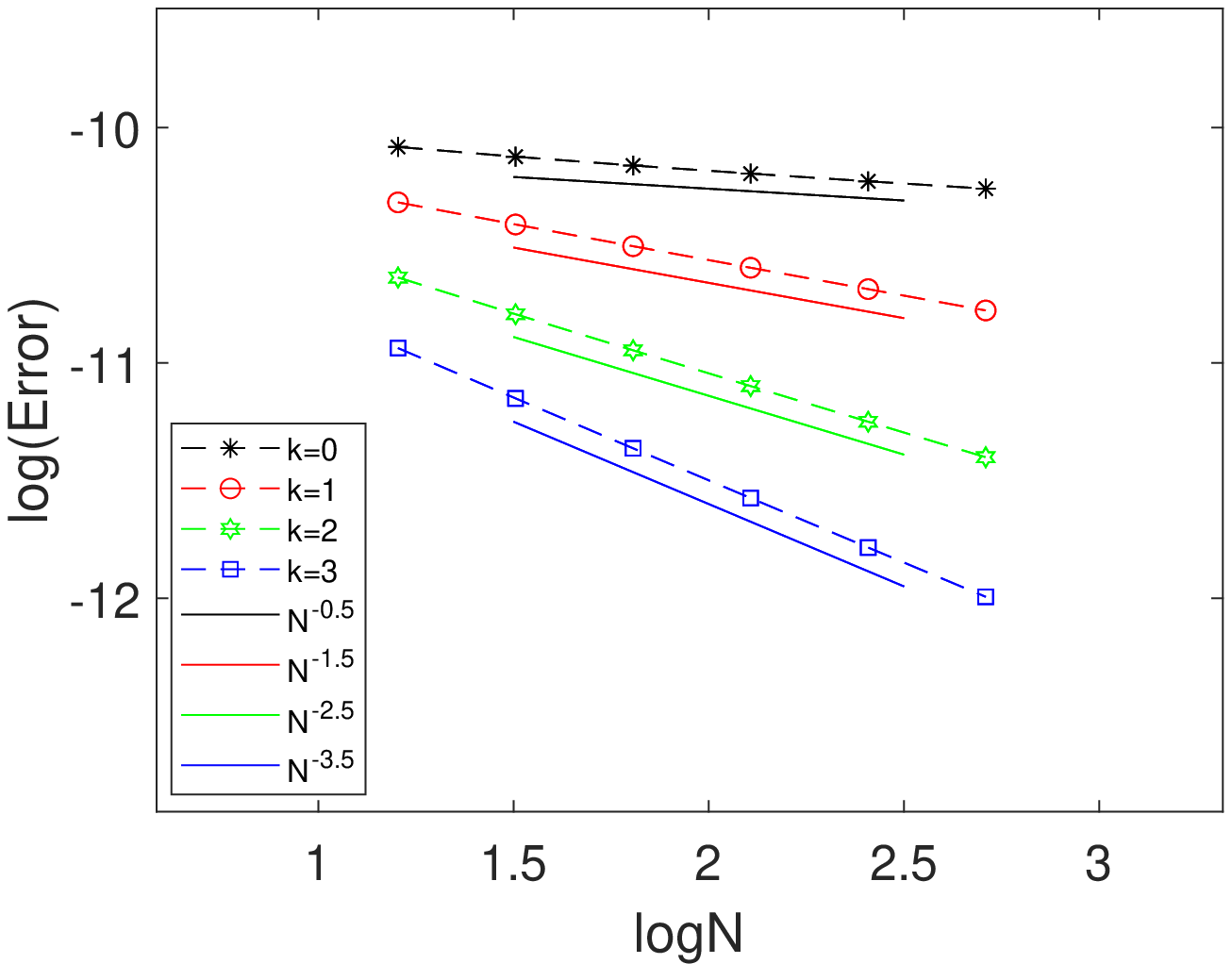}\quad
	\includegraphics[width=1.8in,height=1.8in]{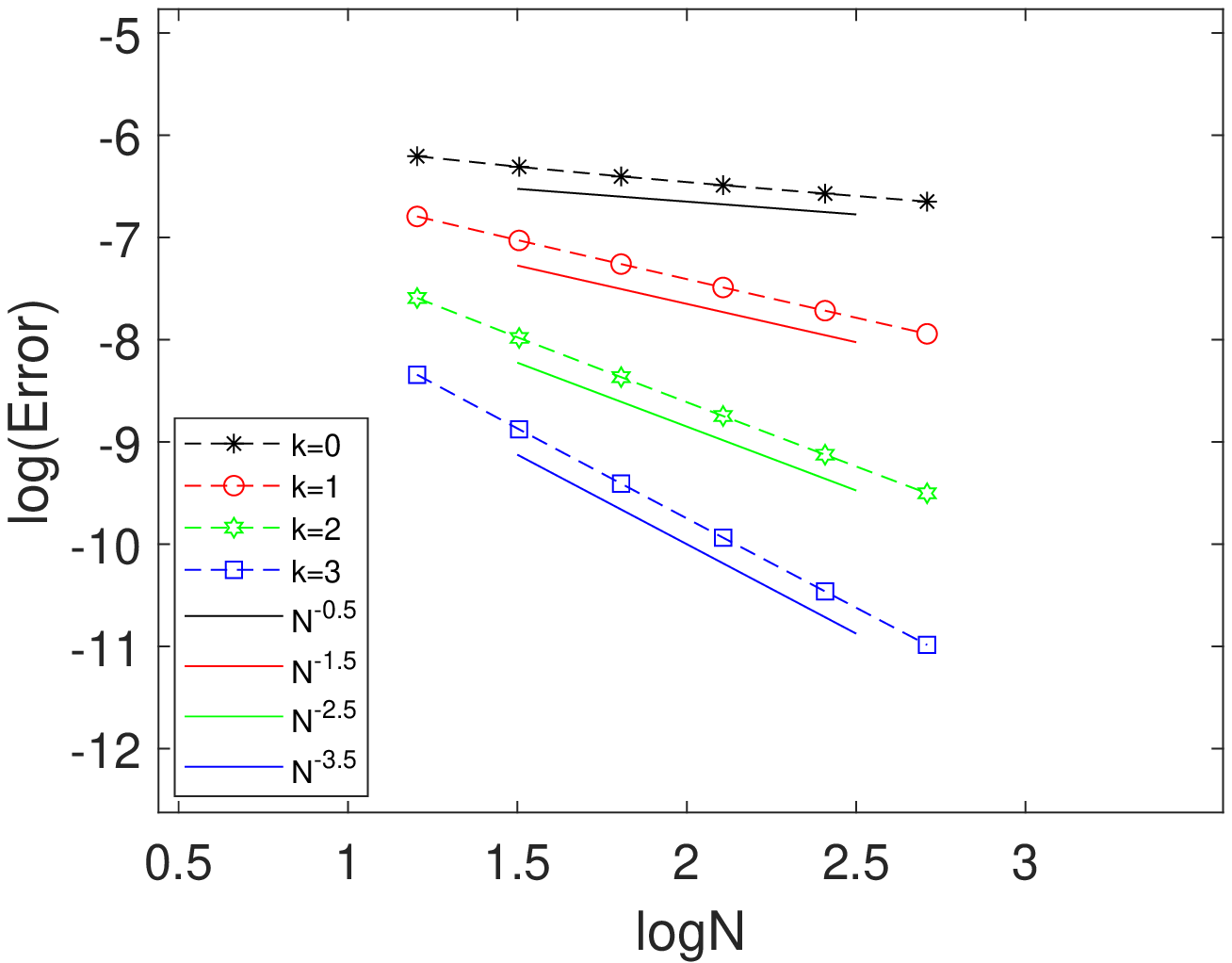}\quad
	\includegraphics[width=1.8in,height=1.8in]{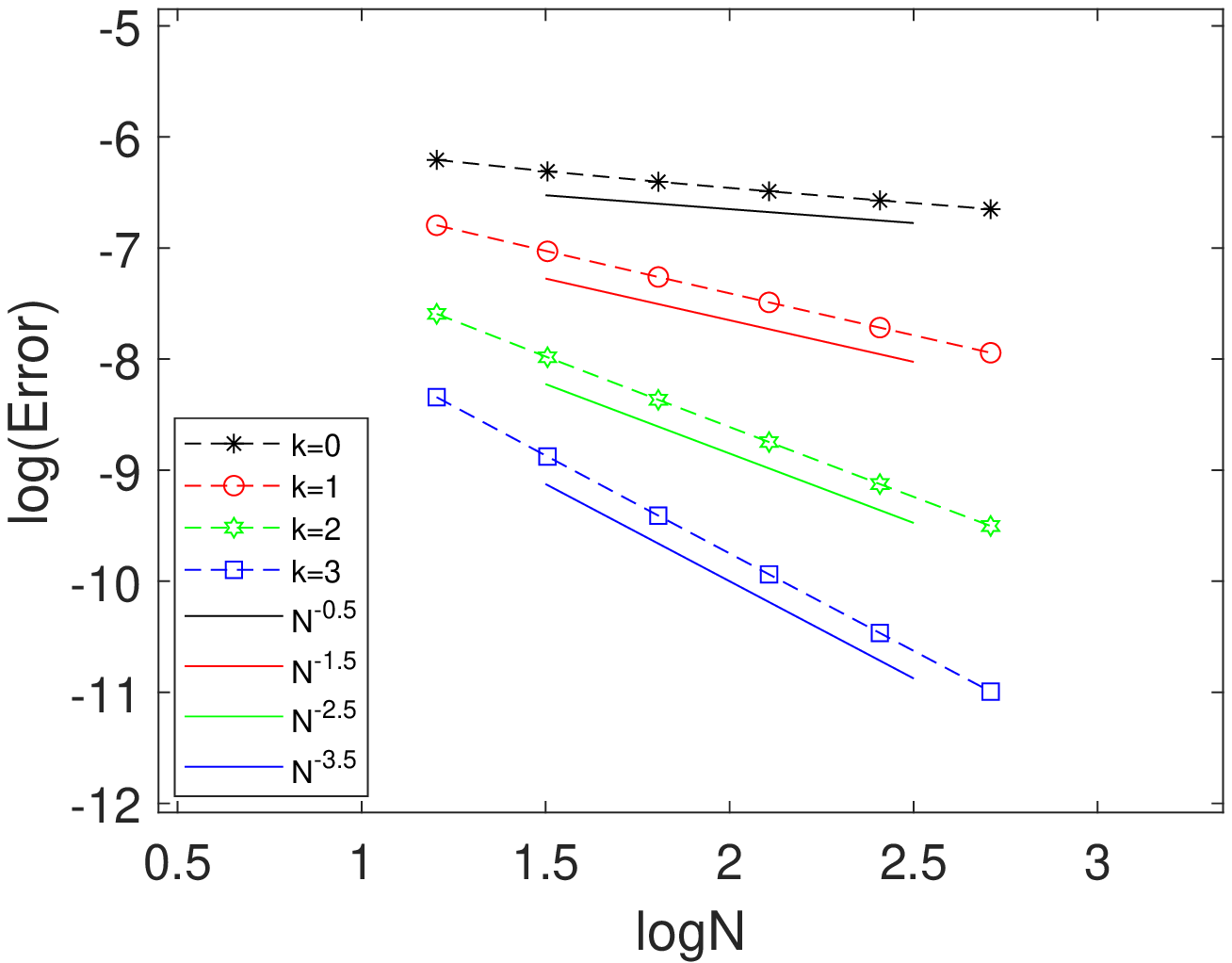}
	\caption{Energy-error:
		$\varepsilon=10^{-4}$ (top); $\varepsilon=10^{-8}$ (bottom);
		left: S mesh; middle: BS mesh; right: B mesh.
	}
	\label{fig:error:energy}
\end{figure}

\begin{figure}[htp]
	\centering
	\includegraphics[width=1.8in,height=1.8in]{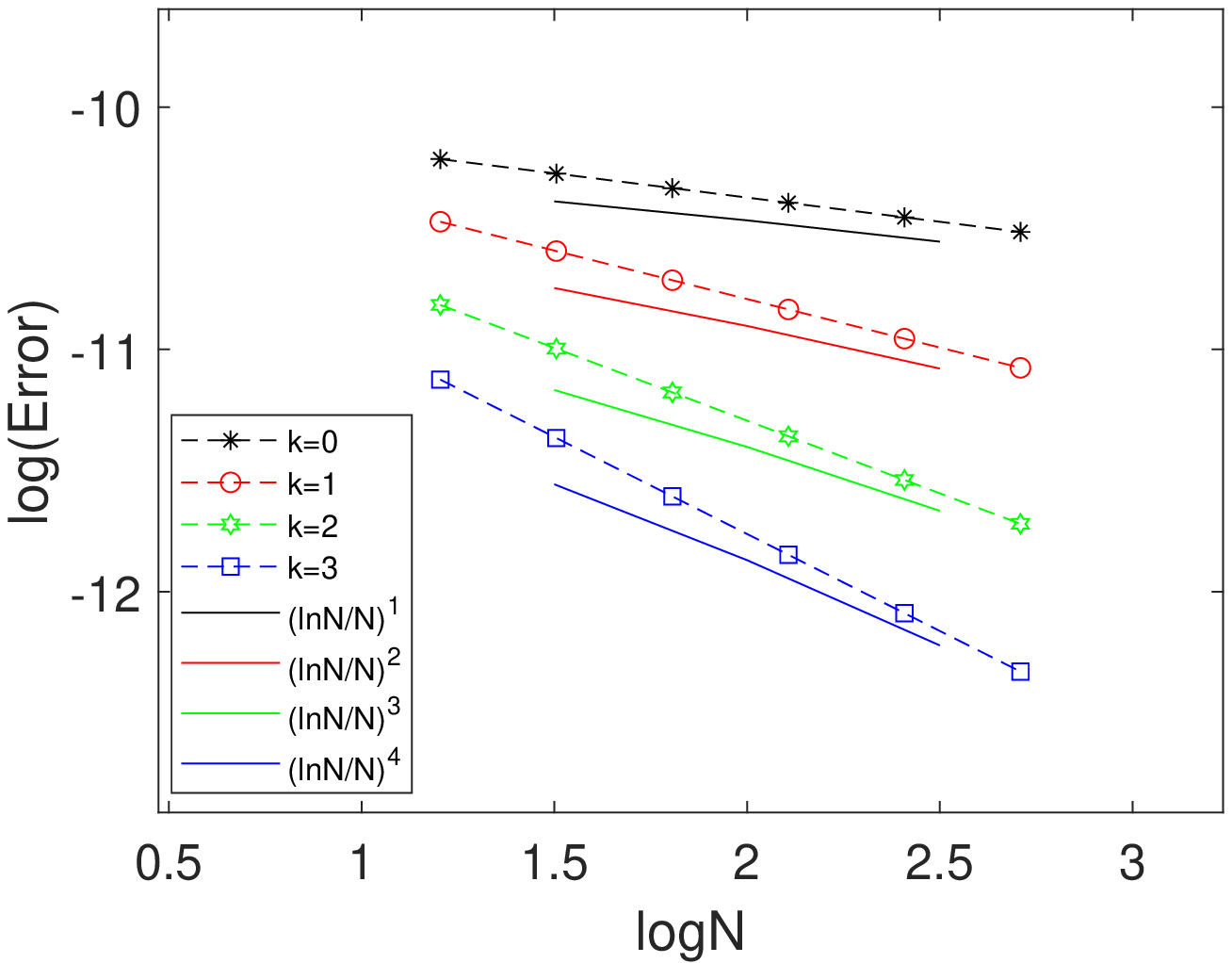}\quad
	\includegraphics[width=1.8in,height=1.8in]{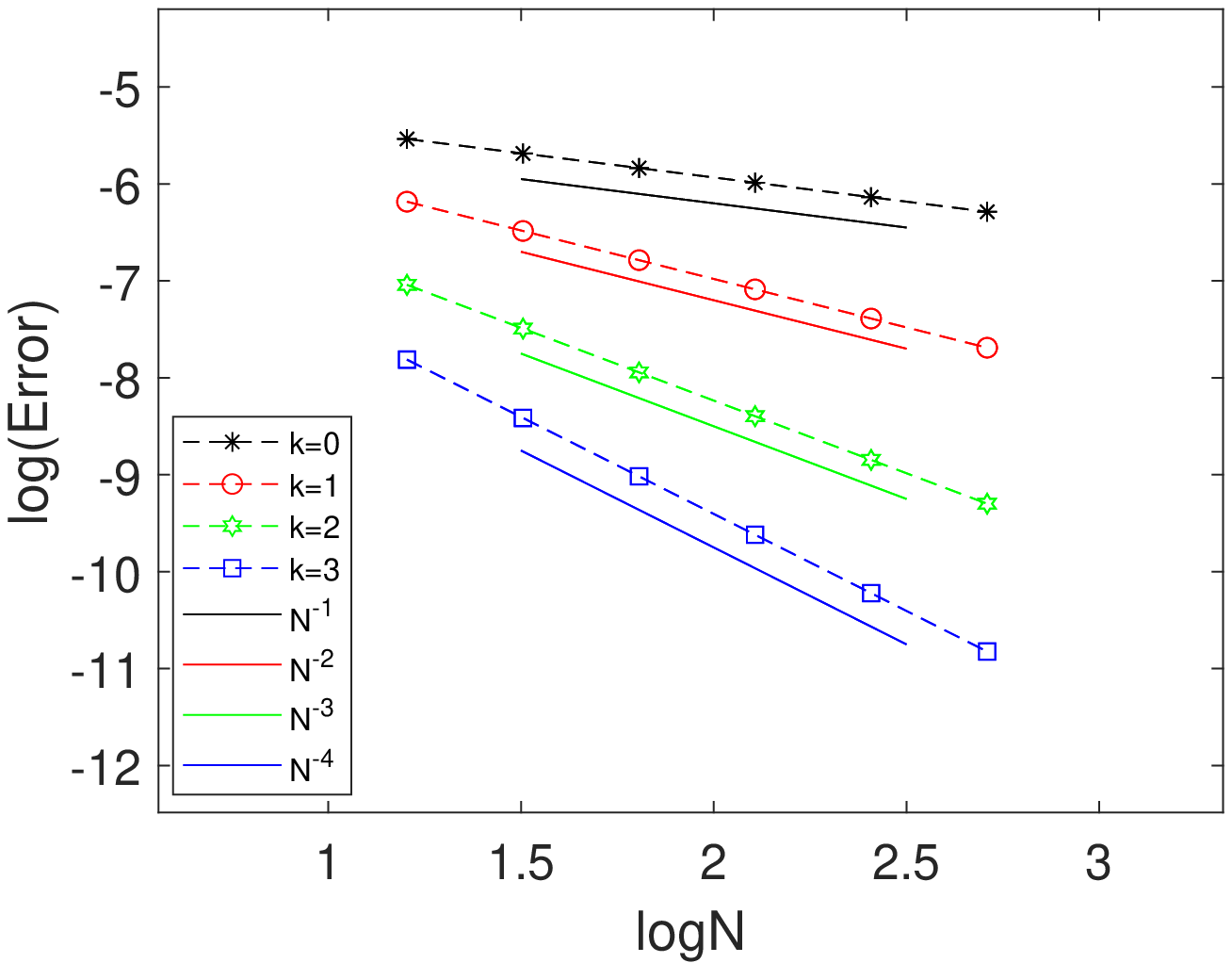}\quad
	\includegraphics[width=1.8in,height=1.8in]{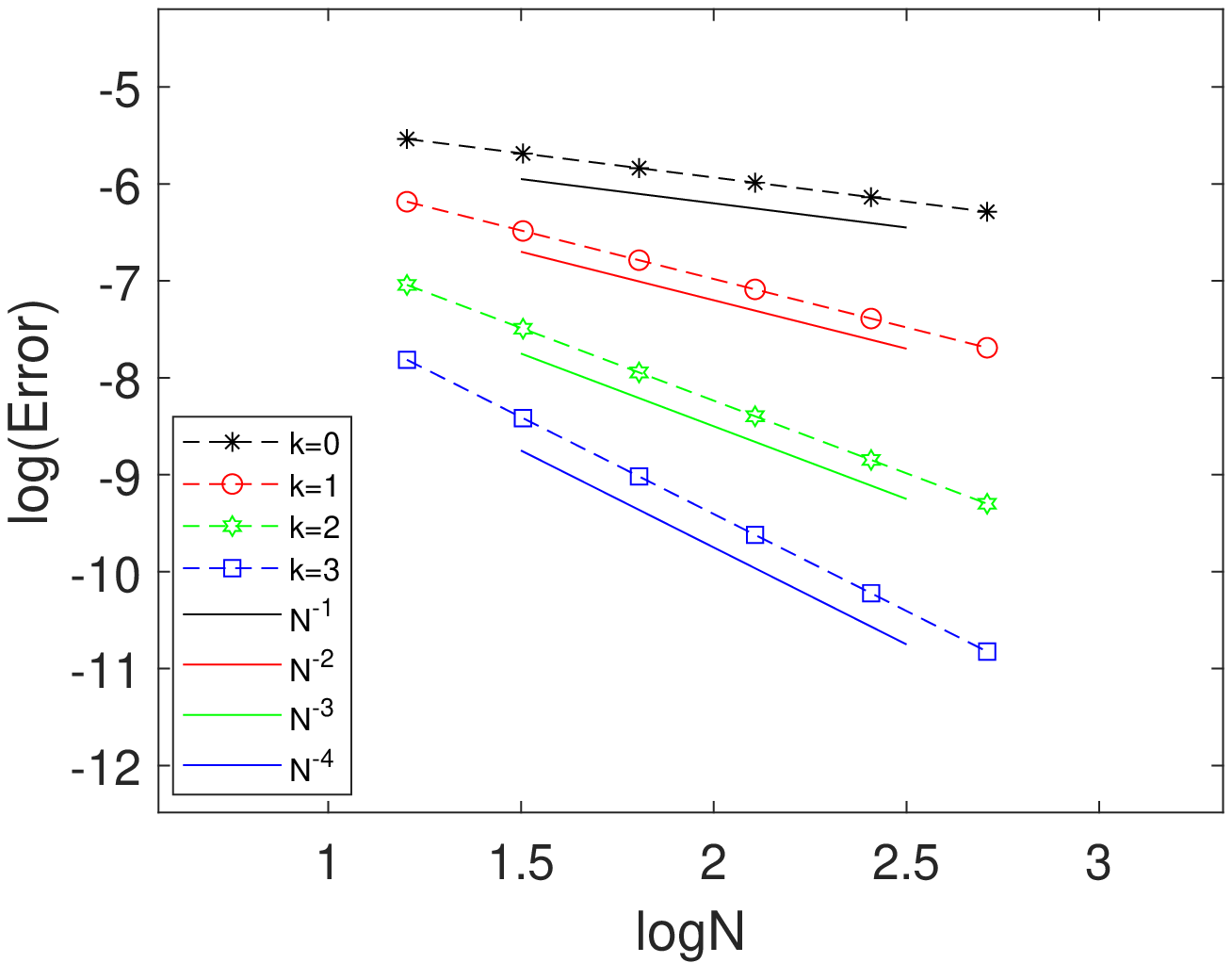}
	\centering
	\includegraphics[width=1.8in,height=1.8in]{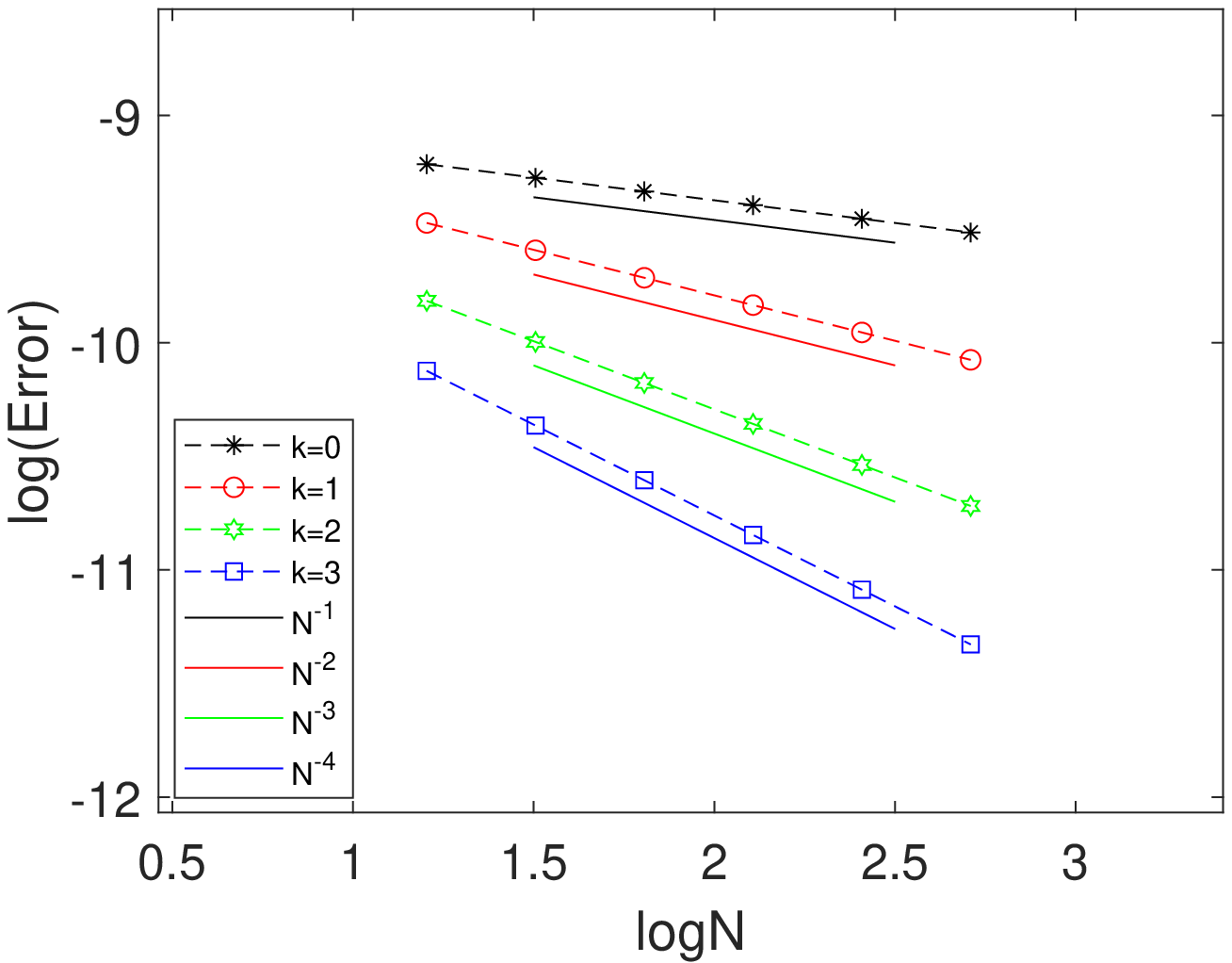}\quad
	\includegraphics[width=1.8in,height=1.8in]{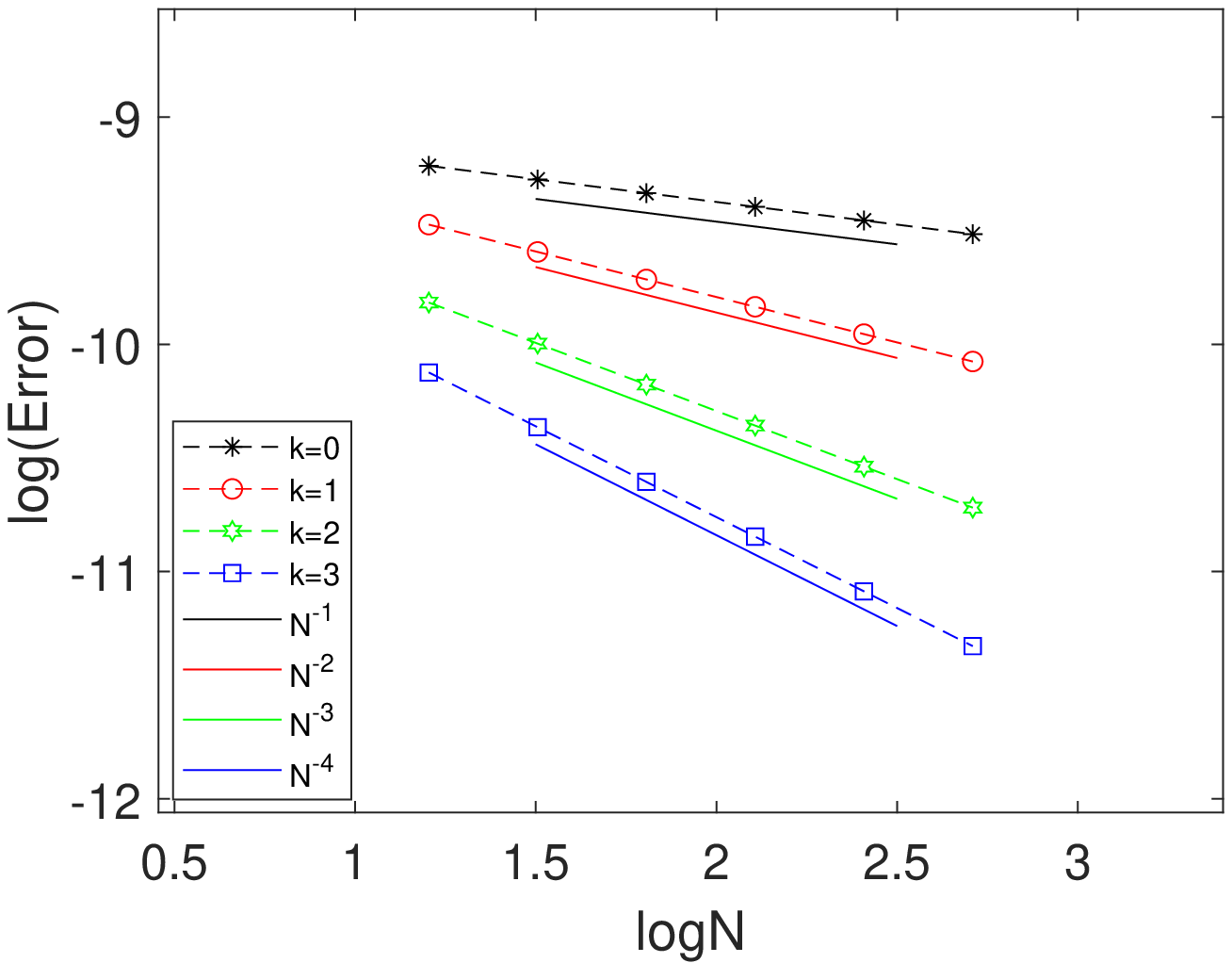}\quad
	\includegraphics[width=1.8in,height=1.8in]{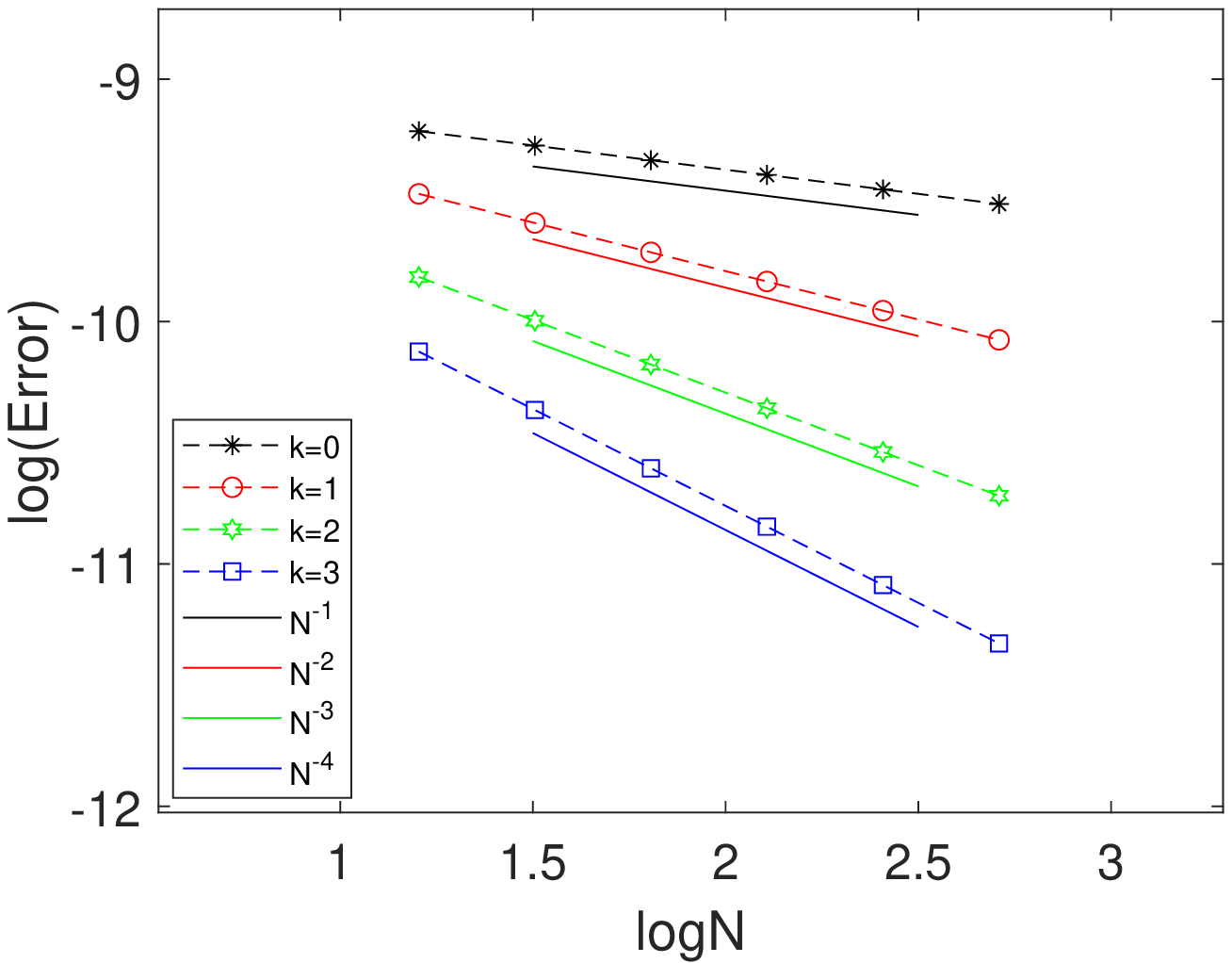}
	\caption{\small
		$L^2$-error $\norm{u-\uph}$:
		$\varepsilon=10^{-4}$ (top); $\varepsilon=10^{-8}$ (bottom);
		left: S mesh; middle: BS mesh; right: B mesh.
	}
	\label{fig:error:u}
\end{figure}

\begin{figure}[htp]
	\centering
	\includegraphics[width=1.8in,height=1.8in]{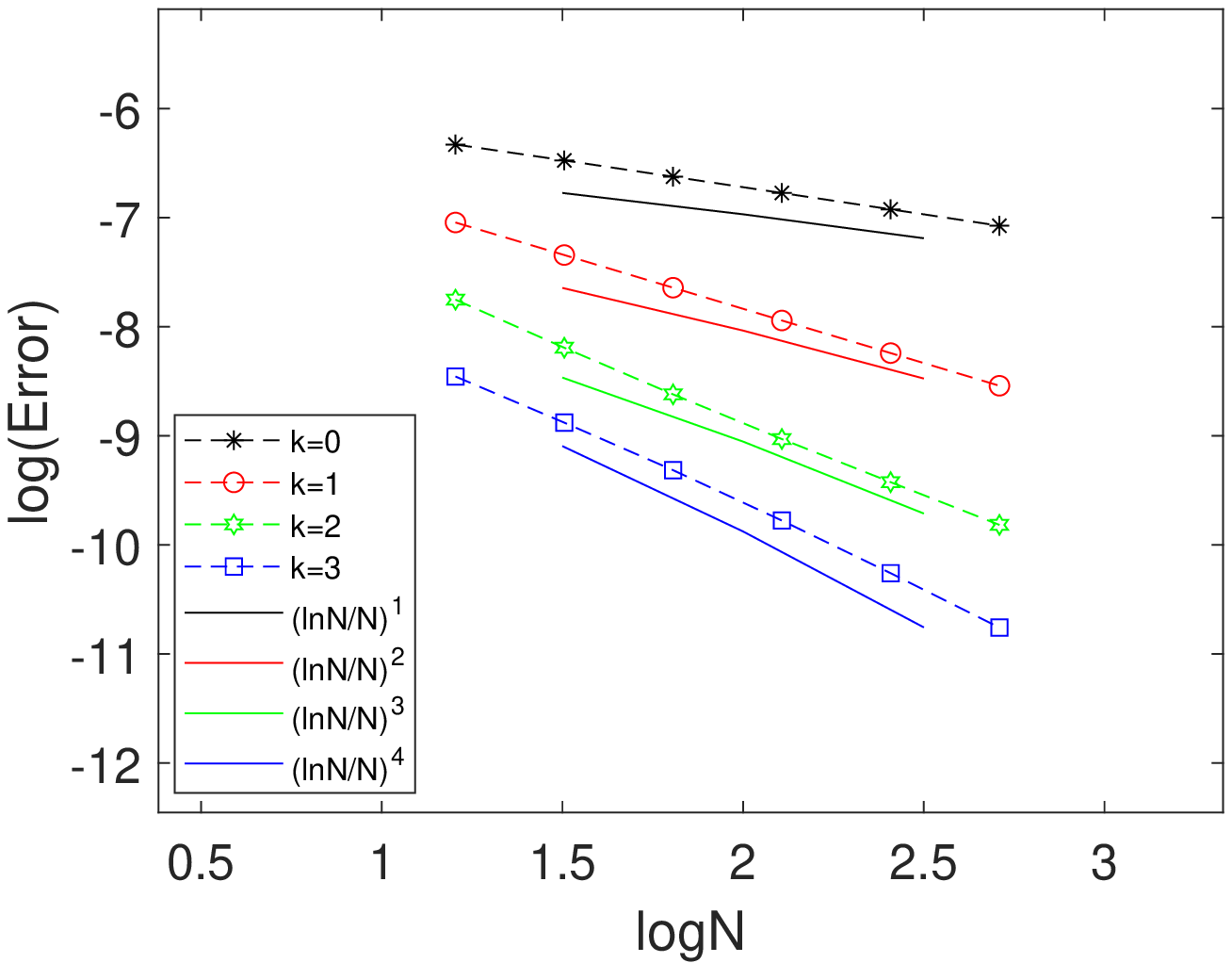}\quad
	\includegraphics[width=1.8in,height=1.8in]{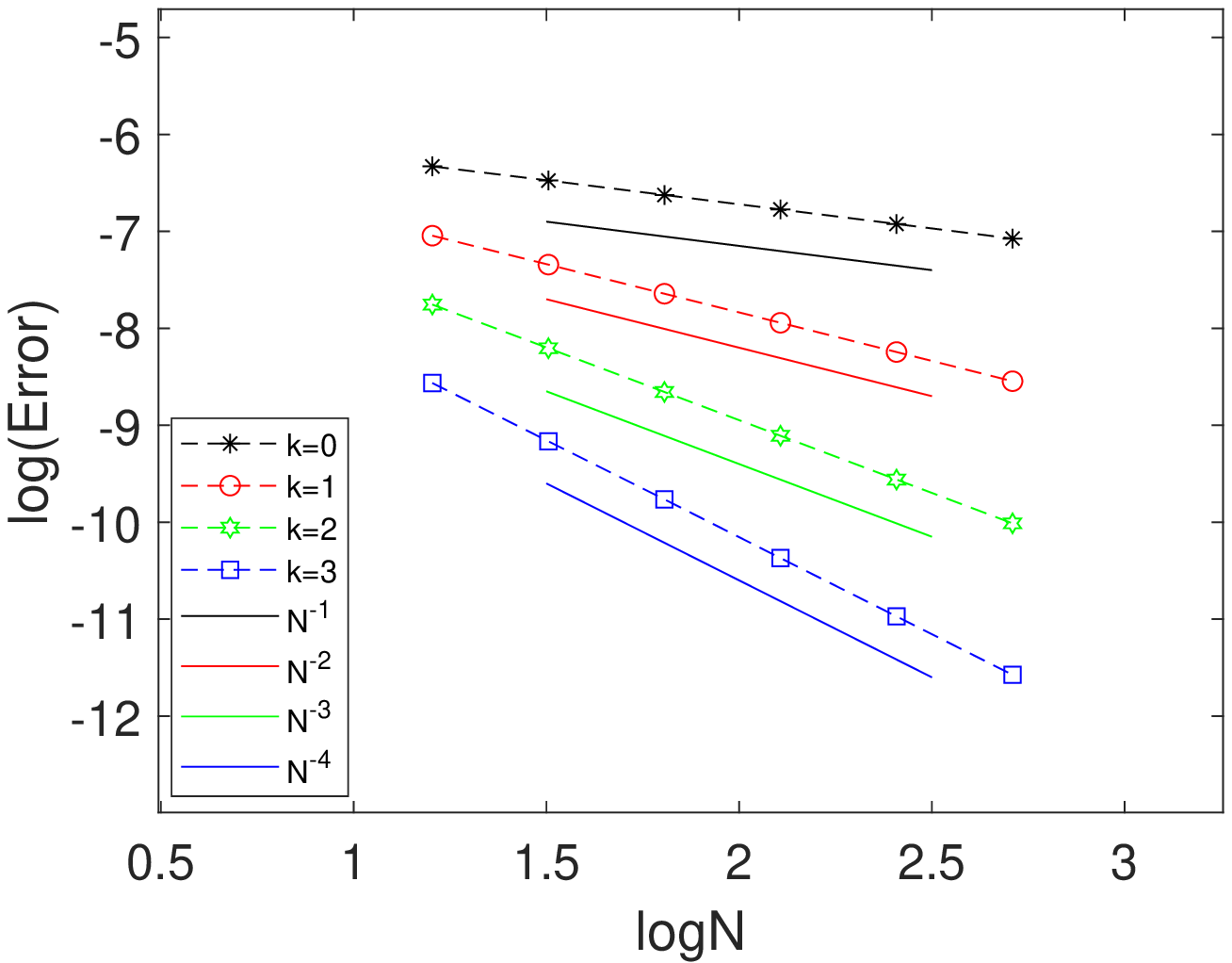}\quad
	\includegraphics[width=1.8in,height=1.8in]{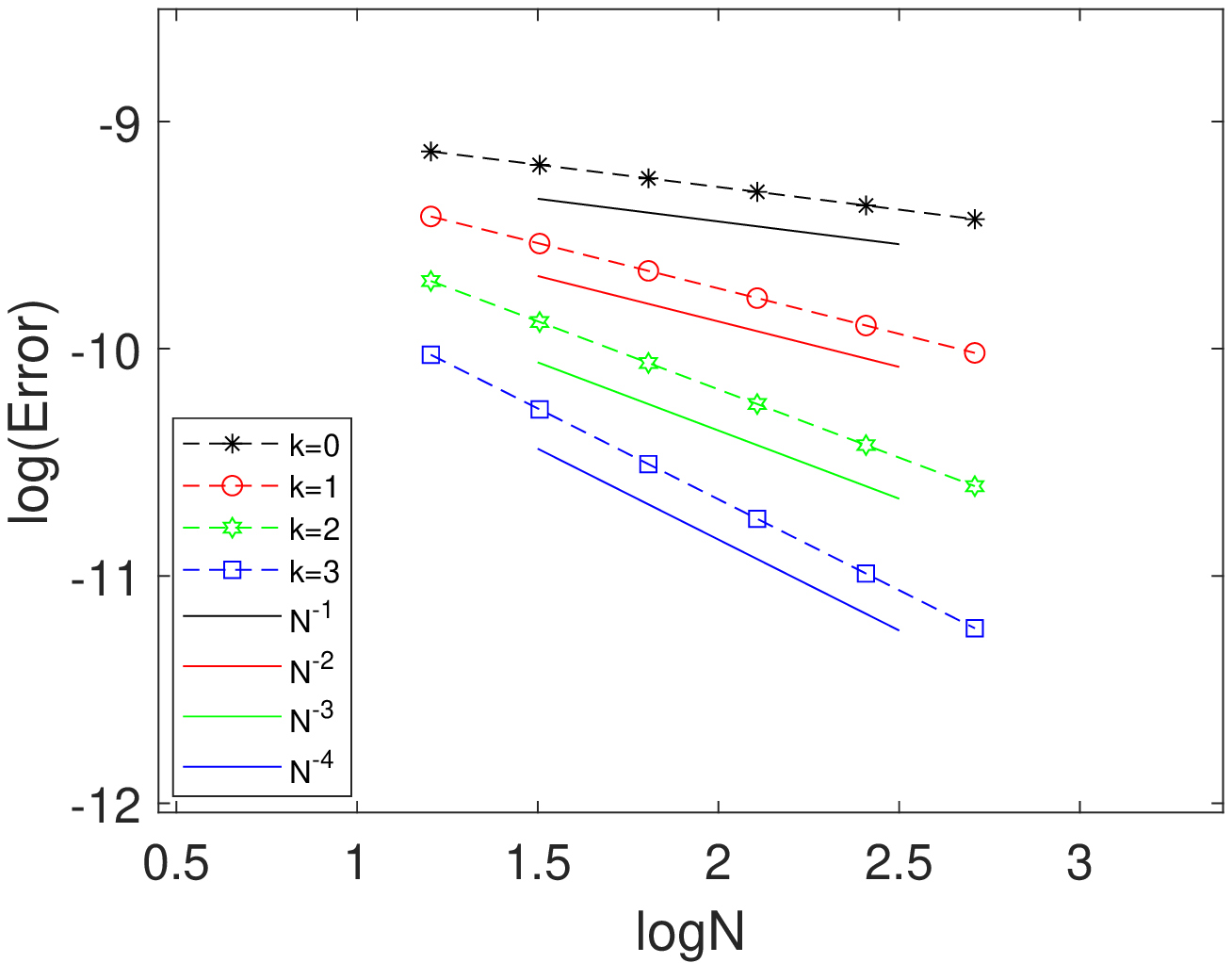}
	\centering
	\includegraphics[width=1.8in,height=1.8in]{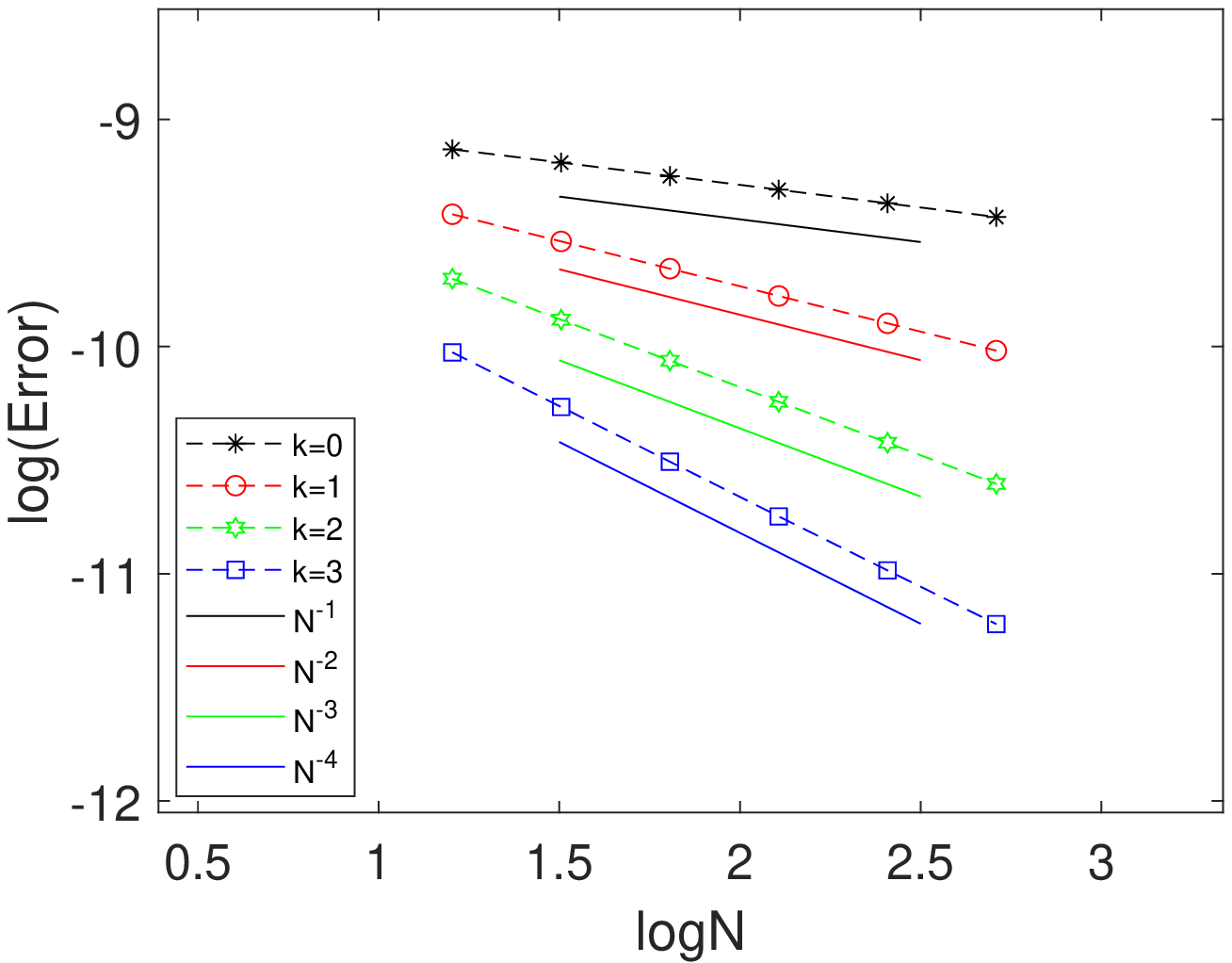}\quad
	\includegraphics[width=1.8in,height=1.8in]{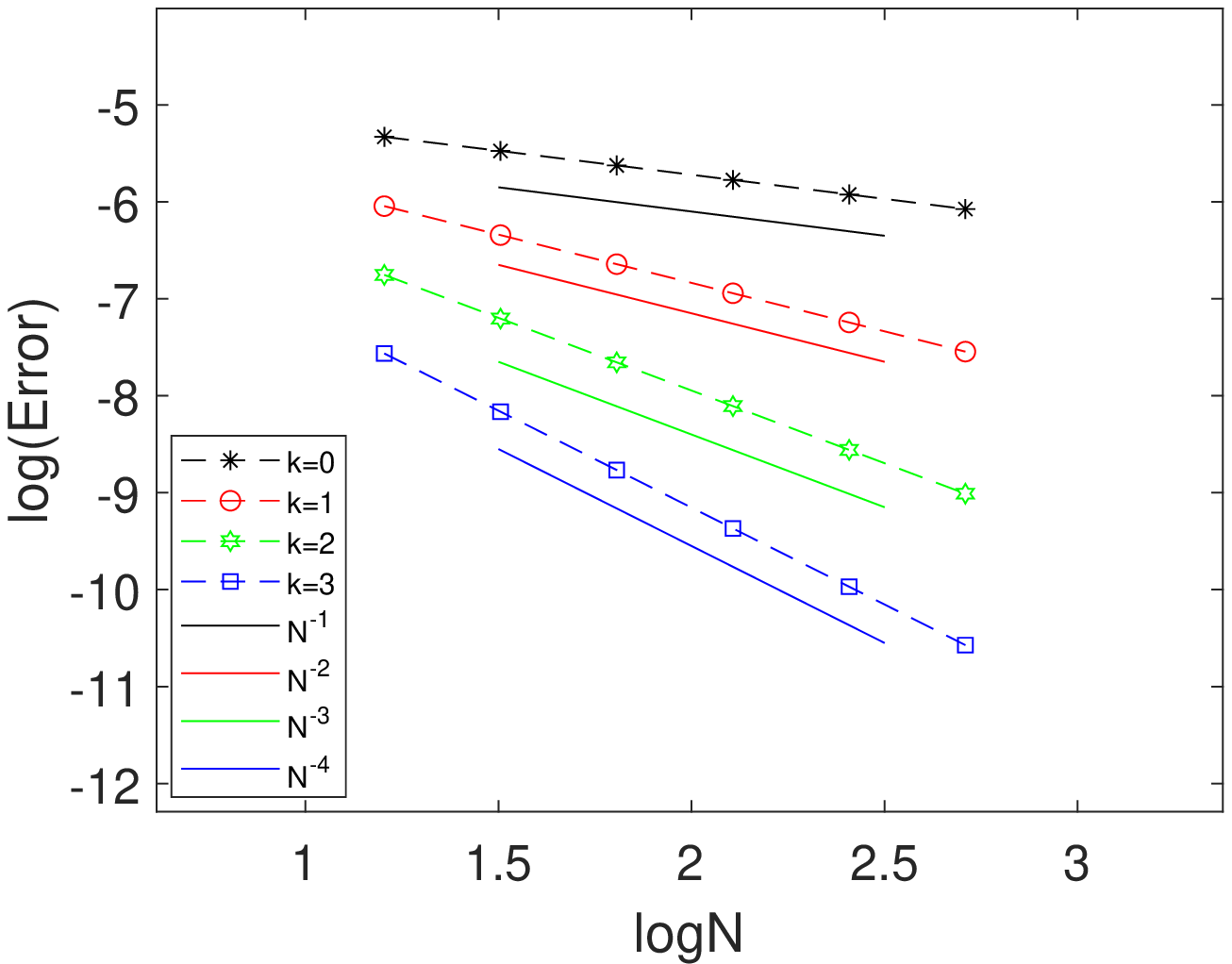}\quad
	\includegraphics[width=1.8in,height=1.8in]{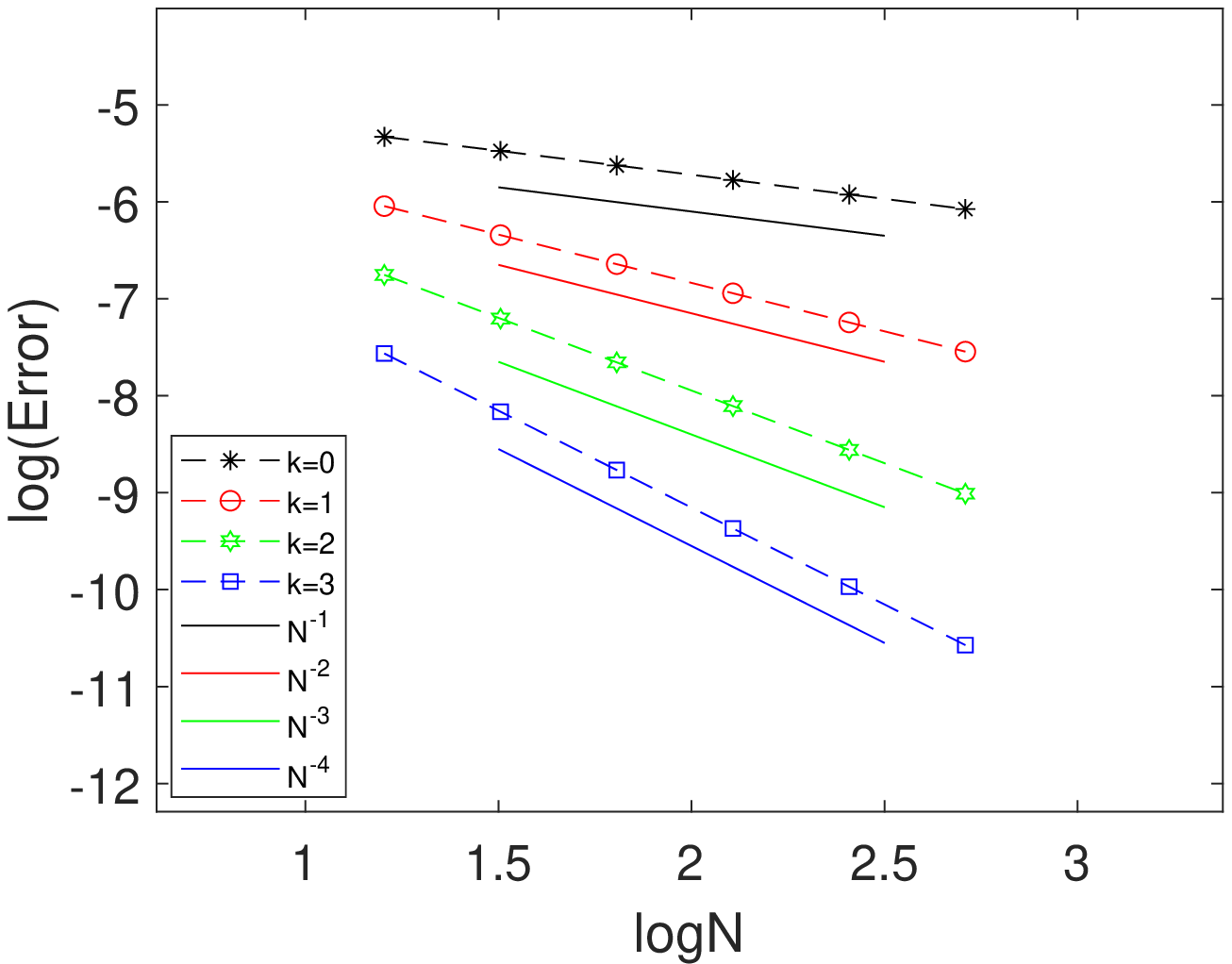}
	\caption{\small
		$L^2$-error $\norm{p-\pph}$:
		$\varepsilon=10^{-4}$ (top); $\varepsilon=10^{-8}$ (bottom);
		left: S mesh; middle: BS mesh; right: B mesh.
	}
	\label{fig:error:p}
\end{figure}

\section{Concluding remarks}
\label{Sec:concluding}

In this study, we considered the local discontinuous Galerkin method
on three typical layer-adapted meshes
for a third-order singularly perturbed problem of convection-diffusion type.
We obtained a quasi-optimal error estimate in the energy norm.
The convergence rate is uniformly valid with respect to a small perturbation parameter.
The theoretical findings provide insights into
the discontinuous Galerkin method 
for higher-order singularly perturbed problems.
Extension of such results to higher odd-order singularly perturbed problems
and higher dimensional singularly perturbed problems
constitute our future work.

\appendix
\section*{Appendix}
\label{appendix}
\setcounter{equation}{0}
\renewcommand{\theequation}{A.\arabic{equation}}


In this part, we provide a technical lemma
used in the proof of our main result.

\begin{lemma} Suppose $(X,Y) \in\mathcal {V}_N^2$ satisfies
	\begin{equation}\label{Variation:form:p:q}
	\dual{Y}{s}_{I_j}
	+\varepsilon(\dual{X}{s^{\prime}}_{I_j}
	-{\hat X}_j s_j^{-}+{\hat X}_{j-1}s_{j-1}^{+})=F_j(s)
	\end{equation}
	in each element $I_j\in \Omega_N$ and
	for any test function $s\in \mathcal{V}_N$,
	where $F_j(s): \mathcal{V}_N\to R$ is a linear functional, and
	\begin{equation*}
	\hat X_j =
	\begin{cases}
	X_j^{+} ,\hspace{2em}j=0,1, \ldots ,N-1,\\
	0,\hspace{3em}j=N.
	\end{cases}
	\end{equation*}
	Then, the local estimate holds
	\begin{equation}\label{relationship:p:q}
	\norm Y_{I_j}
	\le C\varepsilon\left(h_j^{-1}\norm{X}_{I_j}
	+ h_j^{-1/2}|\jump{X}_j|\right)
	+ \frac{|{F_j}(Y)|}{\norm{Y}_{I_j}}
	\end{equation}
	for each element $I_j\in \Omega_N$, 
	where $C>0$ is independent of $\varepsilon$ and $h_j$.
\end{lemma}

\textbf{Proof.} Take $s = Y$ into \eqref{Variation:form:p:q},
use integration by parts, an inverse inequality 
and the Cauchy-Schwarz inequality to get
\begin{align*}\label{simplify equation YY}
\norm{Y}_{I_j}^2
&=\varepsilon(-\dual{X}{Y^{\prime}}_{I_j}
+X_j^{+} Y_j^{-}- X_{j-1}^{+} Y_{j-1}^{+})+F_{j}(Y)
\nonumber\\
& =\varepsilon(\dual{X^{\prime}}{Y}_{I_j}+Y_j^{-}\jump{X}_j)+F_{j}(Y)
\nonumber\\
&\le C\varepsilon(h_j^{-1}\norm{X}_{I_j}\norm{Y}_{I_j}
+h_{j}^{-1/2}|\jump{X}_j|\norm{Y}_{I_j})+|F_{j}(Y)|
\end{align*}
for $j=1,2,\dots,N-1$. Hence,
\begin{align*}
\norm{ Y}_{I_j}  
&\le  C\varepsilon (h_j^{-1}\norm {X}_{I_j}
+h_{j}^{-1/2}|\jump{X}_j|)+ \frac{|F_{j}(Y)|}{\norm{Y}_{I_j}}
\end{align*}
hold for $j=1,2,\dots,N-1$.
Analogously, one can obtain the conclusion for $j=N$
by noticing that $\jump{X}_N=-X_N^{-}$.

\bibliography{YanZhouCheng}
\bibliographystyle{amsplain}

\end{document}